\documentclass[reqno,12pt]{amsart}
\usepackage{amsmath,amsthm,amssymb,amsfonts,amscd}
\usepackage{epsfig}
\usepackage{xypic}
\numberwithin{equation}{section}

\theoremstyle{plain}
\newtheorem{theorem}{Theorem}
\newtheorem{conjecture}[theorem]{Conjecture}
\newtheorem{lemma}[theorem]{Lemma}
\newtheorem{corollary}[theorem]{Corollary}
\newtheorem{proposition}[theorem]{Proposition}
\newtheorem*{theorem*}{Theorem}
\newtheorem*{conjecture*}{Conjecture}

\theoremstyle{definition}
\newtheorem{remark}[theorem]{Remark}
\newtheorem{example}[theorem]{Example}
\newtheorem*{definition}{Definition}

\newcommand{\CC}{{\mathbb{C}}}

\newcommand{\PP}{{\mathbb{P}}}
\newcommand{\QQ}{{\mathbb{Q}}}

\newcommand{\ZZ}{{\mathbb{Z}}}

\def\H{{\mathcal H}}

\begin{document}
\title[Mirror symmetry between orbifold curves and cusp singularities]{Mirror symmetry between orbifold curves and cusp singularities with group action}
\author{Wolfgang Ebeling and Atsushi Takahashi}
\address{Institut f\"ur Algebraische Geometrie, Leibniz Universit\"at Hannover, Postfach 6009, D-30060 Hannover, Germany}
\email{ebeling@math.uni-hannover.de}
\address{
Department of Mathematics, Graduate School of Science, Osaka University, 
Toyonaka Osaka, 560-0043, Japan}
\email{takahashi@math.sci.osaka-u.ac.jp}
\thanks{This work has been supported 
by the DFG-programme SPP1388 ''Representation Theory'' (Eb 102/6-1).
The second named author is also supported 
by Grant-in Aid for Scientific Research 
grant numbers 20360043 from the Ministry of Education, 
Culture, Sports, Science and Technology, Japan. 
}
\begin{abstract} We consider an orbifold Landau-Ginzburg model $(f,G)$, where $f$ is an invertible polynomial in three variables and $G$ a finite group of symmetries of $f$ containing the exponential grading operator, and its Berglund-H\"ubsch transpose $(f^T, G^T)$. We show that this defines a mirror symmetry between orbifold curves and cusp singularities with group action. We define Dolgachev numbers for the orbifold curves and Gabrielov numbers for the cusp singularities with group action. We show that these numbers are the same and that the stringy Euler number of the orbifold curve coincides with the $G^T$-equivariant Milnor number of the mirror cusp singularity.
\end{abstract}
\maketitle
\section*{Introduction}
The Homological Mirror Symmetry Conjecture relates algebraic and symplectic geometry via a categorical approach. 
Kontsevich \cite{Kon} originally formulated his conjecture for Calabi-Yau varieties. 
Inspired by the early constructions of mirror manifolds \cite{BH} and recent work \cite{Kat, Kra, Se, t:1}, we propose the following version of this conjecture.

Following \cite{BH}, we consider an invertible polynomial $f(x,y,z)$  and a finite group $G$ of diagonal symmetries of $f$. We assume that $G$ contains the group $G_0$ generated by the exponential grading operator $g_0$. The construction of Berglund and H\"ubsch associates to the pair $(f,G)$ the pair $(f^T,G^T)$ (the dual group $G^T$ has been described by Krawitz \cite{Kra}). Let $\widetilde{G}$ be the extension of $\CC^\ast$ by $G$ (see Section~\ref{sect:Dol}). The orbifold curve ${\mathcal C}_{(f,G)}:=\left[f^{-1}(0)\backslash\{0\}\left/\widetilde{G} \right.\right]$ is mirror dual to the following data: A function $F : U \to \CC$, defined on a suitably chosen submanifold $U$ of $\CC^3$, given by $F(x,y,z)=x^{\gamma_1'} + y^{\gamma_2'} + z^{\gamma_3'}-xyz$ defining a cusp singularity. The group $G^T$ leaves $F$ invariant and we can consider a crepant resolution $Y \to U/G^T$ given by the $G^T$-Hilbert scheme and the 
proper transform $\widehat{X} \subset Y$ of $X=F^{-1}(0)/G^T \subset U/G^T$ (cf.\ \cite{Se}).

Let ${\rm HMF}_S^{\widetilde{G}}(f)$ be the stable homotopy category of $\widetilde{G}$-graded matrix factorizations of $f$. Let $D^b{\rm Coh}{\mathcal C}_{(f,G)}$ be the derived category of the category of coherent sheaves on ${\mathcal C}_{(f,G)}$.

Let $W(x,y,z)$ be a polynomial which has an isolated singularity at the origin. The directed Fukaya category ${\rm Fuk}^{\to}(W)$ is the categorification of a distinguished basis of (graded) Lagrangian vanishing cycles in the Milnor fiber of $W$. The derived category $D^b{\rm Fuk}^\to(W)$ of this $A_\infty$-category is, as a triangulated category, an invariant of the polynomial $W$. 

We arrive at the following generalization of the Homological Mirror Symmetry Conjecture in \cite{t:1}:
\begin{conjecture}
There should exist triangulated equivalences 
\[ 
\xymatrix{
 {\rm HMF}_S^{\widetilde{G}}(f) \ar[r]^(.4){\sim} \ar@{<->}[d]  & D^b{\rm Fuk}^\to(f^T)//G^T \ar@{<->}[d]\\
D^b{\rm Coh}{\mathcal C}_{(f,G)} \ar[r]^(.4){\sim} & D^b{\rm Fuk}^\to(F)//G^T
}
\]
where the two lines are related by semi-orthogonal decompositions,
$F(x,y,z)=x^{\gamma_1'} + y^{\gamma_2'} + z^{\gamma_3'}-xyz$ is right equivalent  to $f^T(x,y,z)-xyz$, and $-//G^T$ means the smallest triangulated category containing 
the orbit category $-/G^T$ $($cf.\ \cite{As, CM} for orbit categories; see also \cite{Kel}$)$.
\end{conjecture}

The aim of the paper is to provide some evidence for this conjecture. 

We start with a pair $(f,G)$ where $f(x,y,z)$ is an invertible polynomial and $G$ a finite group of symmetries of $f(x,y,z)$ containing the group $G_0$. The orbifold curve ${\mathcal C}_{(f,G)}$ has a finite number of isotropic points with isotropy groups of orders $\alpha_1, \ldots, \alpha_r$. These numbers will be called {\em Dolgachev numbers} of the pair $(f,G)$. We show that, in the case $G=G_0$, they coincide with the orbit invariants of the usual $\CC^\ast$-action on $f^{-1}(0)$.

According to \cite{ET}, the polynomial $f^T(x,y,z)-xyz$ is related to a polynomial $F(x,y,z)=x^{\gamma_1'} + y^{\gamma_2'} + z^{\gamma_3'}-xyz$ of type $T_{\gamma_1',\gamma_2',\gamma_3'}$ by a coordinate change. We show that this coordinate change is $G^T$-equivariant. We define {\em Gabrielov} numbers for the pair $(f^T,G^T)$ by considering the $G^T$-action on the Milnor fiber of $F$. As our main result, we show that we have the following correspondences between invariants of the mirror symmetric pairs: The genus of the orbifold curve ${\mathcal C}_{(f,G)}$ is equal to the number of elements of age 1 in $G^T$ fixing only the origin. The Dolgachev numbers of the pair $(f,G)$ coincide with the Gabrielov numbers of the pair $(f^T,G^T)$. The stringy Euler number of ${\mathcal C}_{(f,G)}$ is equal to the $G^T$-equivariant Milnor number of the cusp singularity defined by $F$. Finally, in some cases the Poincar\'e series of the pair $(f,G_0)$ can be written as a quotient of the characteristic polynomials of the pairs $(f^T, G_0^T)$ and $(F,G_0^T)$.

If $G=G_f^{\rm fin}$ is the maximal group of symmetries of $f$, then $G^T=\{ {\rm id} \}$ and our result specializes to the result of \cite{ET}. In particular, we obtain Arnold's strange duality between the exceptional unimodal singularities. For some cases of $(f,G)$ with $G=G_0$, we recover one direction of the extension of Arnold's strange duality observed in \cite{EW85}. Moreover, we recover Seidel's \cite{Se} example. More generally, for some $(f,G)$ with $G_0 \subset G \subset G_f^{\rm fin}$ we get Efimov's examples \cite{Ef}.

The paper is organized as follows. In the first section, we recall the definition of an invertible polynomial, we consider diagonal symmetry groups of such polynomials, and we recall the definition of the dual group according to \cite{Kra}. In the second section, we collect some properties of the group $G_0^T$.  Dolgachev numbers of the orbifold curves are considered in Section~\ref{sect:Dol}. In Section~\ref{sect:E-function}, we recall the definition of the orbifold E-function and the characteristic polynomial and we compute them for cusp singularities with a group action. These results are used in Section~\ref{sect:Gab} to define the Gabrielov numbers for the pair $(f^T, G^T)$. The main results are presented in Section~\ref{sect:mirror}. Finally we conclude the paper by giving some examples in Section~\ref{sect:Ex}.

\section{Orbifold Landau-Ginzburg models}

\begin{sloppypar}

Let $f(x_1,\dots, x_n)$ be a weighted homogeneous complex polynomial. 
This means that there are positive integers $w_1,\dots ,w_n$ and $d$ such that 
$f(\lambda^{w_1} x_1, \dots, \lambda^{w_n} x_n) = \lambda^d f(x_1,\dots ,x_n)$ 
for $\lambda \in \CC^\ast$. We call $(w_1,\dots ,w_n;d)$ a system of {\em weights}. 
If ${\rm gcd}(w_1,\dots ,w_n,d)=1$, then a system of weights is called {\em reduced}. 
A system of weights which is not reduced is called {\em non-reduced}. 
We shall also consider non-reduced systems of weights in this paper. 

\end{sloppypar}

\begin{definition}
A weighted homogeneous polynomial $f(x_1,\dots ,x_n)$ is called {\em invertible} if 
the following conditions are satisfied:
\begin{enumerate}
\item the number of variables ($=n$) coincides with the number of monomials 
in the polynomial $f(x_1,\dots x_n)$, 
namely, 
\[
f(x_1,\dots ,x_n)=\sum_{i=1}^na_i\prod_{j=1}^nx_j^{E_{ij}}
\]
for some coefficients $a_i\in\CC^\ast$ and non-negative integers 
$E_{ij}$ for $i,j=1,\dots, n$,
\item a system of weights $(w_1,\dots ,w_n;d)$ can be uniquely determined by 
the polynomial $f(x_1,\dots ,x_n)$ up to a constant factor ${\rm gcd}(w_1,\dots ,w_n;d)$, 
namely, the matrix $E:=(E_{ij})$ is invertible over $\QQ$, 
\item $f(x_1,\dots ,x_n)$ and the {\em Berglund--H\"{u}bsch transpose} $f^T(x_1,\dots ,x_n)$ 
of $f(x_1,\dots ,x_n)$ defined by
\[
f^T(x_1,\dots ,x_n):=\sum_{i=1}^na_i\prod_{j=1}^nx_j^{E_{ji}}, 
\]
have singularities only at the origin $0\in\CC^n$ which are isolated.
Equivalently, the {\em Jacobian rings} ${\rm Jac}(f)$ of $f$ and ${\rm Jac}(f^t)$ of $f^t$ defined by
\[  
{\rm Jac}(f):=\CC[x_1,\dots ,x_n]\left/\left(\frac{\partial f}{\partial x_1},\dots ,
\frac{\partial f}{\partial x_n}\right)\right.
\]
\[  
{\rm Jac}(f^T):=\CC[x_1,\dots ,x_n]\left/\left(\frac{\partial f^T}{\partial x_1},\dots ,
\frac{\partial f^T}{\partial x_n}\right)\right.
\]
are both finite dimensional algebras over $\CC$ and $\dim_\CC {\rm Jac}(f), \dim_\CC {\rm Jac}(f^T)\ge 1$.
\end{enumerate}
\end{definition}

\begin{definition}
Let $f(x_1,\dots ,x_n)=\sum_{i=1}^na_i\prod_{j=1}^nx_j^{E_{ij}}$ be an invertible polynomial.
The {\em canonical system of weights} $W_f$ is the system of weights 
$(w_1,\dots ,w_n;d)$ given by the unique solution of the equation
\begin{equation*}
E
\begin{pmatrix}
w_1\\
\vdots\\
w_n
\end{pmatrix}
={\rm det}(E)
\begin{pmatrix}
1\\
\vdots\\
1
\end{pmatrix}
,\quad 
d:={\rm det}(E).
\end{equation*}
\end{definition}

\begin{remark}
It follows from Cramer's rule that $w_1,\dots ,w_n$ are positive integers.
Note that the canonical system of weights is in general non-reduced.
\end{remark}

\begin{definition}
Let $f(x_1,\dots, x_n)$ be an invertible polynomial and $W_f=(w_1,\dots ,w_n;d)$ 
the canonical system of weights attached to $f$. 
Define \[ c_f:= {\rm gcd}(w_1, \dots  ,w_n,d). \]
\end{definition}

\begin{definition}
Let $f(x_1,\dots ,x_n)=\sum_{i=1}^na_i\prod_{j=1}^nx_j^{E_{ij}}$ be an invertible polynomial.
The {\em maximal abelian symmetry group} $G_f$ of $f$ is the abelian group defined by 
\[
G_f=\left\{(\lambda_1,\dots ,\lambda_n)\in(\CC^\ast)^n \, \left| \,
\prod_{j=1}^n \lambda_j ^{E_{1j}}=\dots =\prod_{j=1}^n\lambda_j^{E_{nj}}\right\} \right..
\]
\end{definition}
Note that the polynomial $f$ is homogeneous with respect to the natural action of 
$G_f$ on the variables. Namely, we have 
\[
f(\lambda_1 x_1, \dots, \lambda_n x_n) = \lambda f(x_1,\dots ,x_n)
\]
for $(\lambda_1,\dots ,\lambda_n)\in G_f$ where 
$\lambda :=\prod_{j=1}^n \lambda_j ^{E_{1j}}=\dots =\prod_{j=1}^n\lambda_j^{E_{nj}}$.

\begin{definition}
Let $f(x_1,\dots, x_n)$ be an invertible polynomial and $W_f=(w_1,\dots ,w_n;d)$ 
the canonical system of weights attached to $f$. Set
\[ q_i := \frac{w_i}{d}, \quad i=1, \ldots , n.\]
\end{definition}

Note that we have
\begin{equation*}
E
\begin{pmatrix}
q_1\\
\vdots\\
q_n
\end{pmatrix}
=
\begin{pmatrix}
1\\
\vdots\\
1
\end{pmatrix}.
\end{equation*}

\begin{definition} Let $f(x_1,\dots, x_n)$ be an invertible polynomial. Let $G^{\rm fin}_f$ be the group of diagonal symmetries of $f$, i.e. 
\[
G^{\rm fin}_f=\left\{(\lambda_1,\dots ,\lambda_n)\in G_f \, \left| \, \prod_{j=1}^n \lambda_j ^{E_{1j}}=\dots =\prod_{j=1}^n\lambda_j^{E_{nj}}=1 \right\} \right..
\]
\end{definition}

Let ${\bf e}[ - ] = e^{2 \pi \sqrt{-1} \cdot  -}$.
Note that $G^{\rm fin}_f$ always contains the exponential grading operator 
\[
g_0:=({\bf e}[q_1], \ldots , {\bf e}[q_n]).
\]

The group $G^{\rm fin}_f$ can be described as follows (cf.\ \cite{Kra}). Write
\[
E^{-1} = \begin{pmatrix}
a_{11} & \cdots & a_{1n}\\
\vdots & \ddots & \vdots \\
a_{n1} & \cdots & a_{nn}
\end{pmatrix} = 
\begin{pmatrix} \sigma_1 & \cdots &  \sigma_n \end{pmatrix} \mbox{ with column vectors }
\sigma_j = \begin{pmatrix} 
a_{1j} \\ \vdots \\ a_{nj} \end{pmatrix}.
\]
Each $\sigma_j$ defines an element of $G^{\rm fin}_f$ by
\[ 
\sigma_j= ({\bf e}[a_{1j}], \ldots, {\bf e}[a_{nj}]) \in (\CC^\ast)^n.
\]
One can easily see that each element $g \in G^{\rm fin}_f$ can be written $g = \prod_{j=1}^n \sigma_j^{s_j}$ for some integers $s_j$, $j=1, \ldots, n$. In particular, $g_0=\prod_{j=1}^n \sigma_j$.

Let $G \subset G^{\rm fin}_f$ be a subgroup. M.~Krawitz \cite{Kra} defines  a dual group $G^T$ as follows. 
Write
\[ 
E^{-1} = \begin{pmatrix} \rho_1\\ \vdots \\  \rho_n \end{pmatrix} \mbox{ with row vectors }
\rho_i = (a_{i1},  \ldots , a_{in}).
\]
As above,  each $\rho_i$ can be considered as an element of $G^{\rm fin}_{f^T}$ as follows:
\[ 
\rho_i= ({\bf e}[a_{i1}], \ldots, {\bf e}[a_{in}]) \in (\CC^\ast)^n.
\]
These elements generate $G^{\rm fin}_{f^T}$.

\begin{definition} The {\em dual group} $G^T$ is defined by
\[
G^T := \left\{ \prod_{i=1}^n \rho_i^{r_i} \, \left| \, (r_1, \ldots, r_n) E^{-1} \begin{pmatrix} s_1 \\ \vdots \\ s_n \end{pmatrix} \in \ZZ \mbox{ for all } \prod_{j=1}^n \sigma_j^{s_j} \in G \right\} \right. .
\]
\end{definition}

\begin{definition} Let $f$ be an invertible polynomial and let $G \subset G^{\rm fin}_f$ be a subgroup. The pair $(f,G)$ is called an {\em orbifold Landau-Ginzburg model}.
\end{definition}

\begin{proposition} \label{Prop:Gfin/G}
Let $G \subset G^{\rm fin}_f$ be a subgroup. Then there is a natural group isomorphism 
\[ {\rm Hom}(G^{\rm fin}_f/G, \CC^\ast) \simeq G^T. \]
\end{proposition}

\begin{proof}
Apply the functor ${\rm Hom}( -, \CC^\ast)$ to the short exact sequence
\[ 
\begin{CD}
\{ 1 \}  \longrightarrow G \longrightarrow G^{\rm fin}_f \longrightarrow G^{\rm fin}_f/G \longrightarrow \{ 1 \} .
\end{CD}
\]
Then we obtain the short exact sequence
\[ \{ 1 \} \longrightarrow {\rm Hom}(G^{\rm fin}_f/G, \CC^\ast) \longrightarrow {\rm Hom}(G^{\rm fin}_f, \CC^\ast) \longrightarrow {\rm Hom}(G, \CC^\ast) \longrightarrow \{ 1\} . \]
The group ${\rm Hom}(G^{\rm fin}_f, \CC^\ast)$ is isomorphic to the group $G^{\rm fin}_{f^T}$. Let $g^T=\prod_{i=1}^n \rho_i^{r_i} \in G^{\rm fin}_{f^T}$. Then $g^T \in G^T$ if and only if 
\[ (r_1, \ldots, r_n) E^{-1} \begin{pmatrix} s_1 \\ \vdots \\ s_n \end{pmatrix} \in \ZZ \mbox{ for all } \prod_{j=1}^n \sigma_j^{s_j} \in G. \]
But this is equivalent to the fact that $g^T$ lies in the kernel of the map ${\rm Hom}(G^{\rm fin}_f, \CC^\ast) \longrightarrow {\rm Hom}(G, \CC^\ast)$.
\end{proof}

\begin{remark} As a corollary of Proposition~\ref{Prop:Gfin/G} we obtain 
\[  (G^T)^T =G \]
(cf.\ \cite[Lemma 3.3]{Kra}). In particular, $\{ 1 \}^T= G^{\rm fin}_f$.
\end{remark}

\section{The group $G_0^T$} \label{sect:g0}
We now consider the group $G_0^T$ where $G_0:=\langle g_0 \rangle$.

By \cite{Kra} we have 
\[
G_0^T =  \left\{ \prod_{i=1}^n \rho_i^{r_i} \, \left| \, \sum_{i=1}^n r_i q_i \in \ZZ \right\} \right. = {\rm SL}_n(\CC) \cap G^{\rm fin}_{f^T}.
\]

\begin{proposition} \label{Prop:ordJT}
We have
\[ |G_0^T | = c_f. \]
\end{proposition}

\begin{proof}
The group $G_0$ has order $d/c_f$.
By \cite{Kre} we have 
\[ |G^{\rm fin}_{f^T}| = |G^{\rm fin}_f| = d = \det E. \]
Therefore the statement follows from Proposition~\ref{Prop:Gfin/G}.
\end{proof}

According to \cite{KS}, an invertible polynomial $f$ is a (Thom-Sebastiani) sum of invertible polynomials (in groups of different variables) of the following types:
\begin{enumerate}
\item[1)] $x_1^{a_1}$ (Fermat type);
\item[2)] $x_1^{a_1}x_2 + x_2^{a_2}x_3 + \ldots + x_{m-1}^{a_{m-1}}x_m + x_m^{a_m}$ (chain type; $m\ge 2$);
\item[3)] $x_1^{a_1}x_2 + x_2^{a_2}x_3 + \ldots + x_{m-1}^{a_{m-1}}x_m + x_m^{a_m}x_1$ (loop type; $m\ge 2$).
\end{enumerate}

\begin{proposition} \label{Prop:genJT}
\begin{itemize}
\item[{\rm (i)}] Let $f(x_1, \ldots, x_n)=x_1^{a_1}x_2 + x_2^{a_2}x_3 + \ldots + x_{n-1}^{a_{n-1}}x_n + x_n^{a_n}$ be of chain type. Then $G_0^T$ is a cyclic group of order $c_f$ which is generated by $\rho_n^{d/c_f}$. This element acts on the first coordinate $x_1$ by multiplying it by ${\bf e}[1/c_f]$. 
\item[{\rm (ii)}] Let $f(x_1, \ldots, x_n)=x_1^{a_1}x_2 + x_2^{a_2}x_3 + \ldots + x_{n-1}^{a_{n-1}}x_n + x_n^{a_n}x_1$ be of loop type. Then $G_0^T$ is a cyclic group of order $c_f$ which is generated by $\rho_i^{d/c_f}$ for any $i=1, \ldots , n$. Here $\rho_i^{d/c_f}$ acts on the $i$-th coordinate $x_i$ by multiplying it by ${\bf e}[1/c_f]$. 
\end{itemize}
\end{proposition}

\begin{proof} Note that if $f$ is of chain or loop type then $f^T$ is of the same type. The statement then follows from \cite{Kre} where it shown that $G^{\rm fin}_f$ is a cyclic group in these cases and generators for this group are given.
\end{proof}

Let $G \subset {\rm GL}_n(\CC)$ be a finite subgroup and $g \in G$ an element. The age of $g$ is defined as follows \cite{IR}: Let $r$ be the order of $g$. Choose a basis of eigenvectors diagonalizing the matrix $g$ , giving
\[ g={\rm diag}({\bf e}[a_1/r], \ldots, {\bf e}[a_n/r]) \quad \mbox{with } 0 \leq a_i < r. \]
The element $g$ is written $g=\frac{1}{r}(a_1, \ldots, a_n)$. We define the {\em age} of $g$ to be the rational number
\[ {\rm age}(g) := \frac{1}{r}\sum_{i=1}^n a_i. \] 
If $g \in {\rm SL}_n(\CC)$, then this number is an integer. The elements of age 1 are also called {\em junior elements}.

\begin{remark}
For the exponential grading operator $g_0 \in G^{\rm fin}_f$ we have ${\rm age}(g_0)=\sum_{i=1}^n q_i$.
\end{remark}

For $g \in G$, we denote by ${\rm Fix}\, g$ the fixed locus of $g$.

\begin{definition} Let $G \subset {\rm GL}_n(\CC)$ be a finite subgroup. The number of elements  $g \in G$ of age 1 with ${\rm Fix}\, g = \{ 0 \}$ is denoted by $j_G$.
\end{definition}

\begin{remark} \label{rem:fix}
Let $G \subset {\rm SL}_3(\CC)$ be a finite subgroup.
The number of $g \in G$ with ${\rm Fix}\, g = \{ 0 \}$ is equal to $2j_G$. For, if $g=\frac{1}{r}(a_1,a_2,a_3)$ is an element of age 1 with $0<a_i <1$, $i=1,2,3$, then $g^{-1}=\frac{1}{r}(1-a_1,1-a_2,1-a_3)$ is an element of age 2.
\end{remark}

Let $f(x_1, \ldots , x_n)$ be an invertible polynomial. Let $\Omega^p(\CC^n)$ be the algebra of regular $p$-forms on $\CC^n$. Consider the algebra $\Omega_f:= \Omega^n(\CC^n)/(df \wedge \Omega^{n-1}(\CC^n))$. This algebra is $\QQ$-graded  with respect to the system of (fractional) weights $(q_1, \ldots, q_n)$. Let $\Omega_{f,1}$ be the subspace of $\Omega_f$ of elements of degree 1. A subgroup $G \subset G^{\rm fin}_f$ induces an action on $\Omega_{f,1}$ via its action on the coordinates. We denote by $(\Omega_{f,1})^G$ the subspace of invariants under this group action.

\begin{proposition} \label{Prop:g=j}
Let $f(x_1, \ldots , x_n)$, $n \geq 3$, be an invertible polynomial and $G \subset G^{\rm fin}_f$ be a subgroup with $G_0 \subset G$. Then we have
\[ \dim_\CC (\Omega_{f,1})^G = j_{G^T}. \]
\end{proposition}

\begin{proof} Let $g=\prod_{i=1}^n \rho_i^{r_i}$ be an element of $G_0^T$. Then $g$ acts on $\CC^n$ by
\[ g= {\rm diag} \left({\bf e}\left[\sum_{i=1}^n r_ia_{i1}\right], \ldots , {\bf e}\left[\sum_{i=1}^n r_ia_{in}\right]\right). \]
Assume that $0 \leq \sum_{i=1}^n r_ia_{ij} <1$ for all $j=1, \ldots, n$. Then the age of $g$ is given by the formula
\begin{equation}
 {\rm age}(g) = (r_1, \ldots , r_n) E^{-1} \begin{pmatrix} 1 \\ \vdots \\ 1 \end{pmatrix} = \sum_{i=1}^n r_i q_i.  \label{Eq:age}
\end{equation}

There is a mapping ${\rm Hom}(G, \CC^\ast) \simeq G^T \to \Omega_f$ defined as follows. Let $g=\prod_{i=1}^n \rho_i^{r_i} \cdot g_0^T \in {\rm Hom}(G, \CC^\ast)$ be represented in the form that $0 \leq \sum_{i=1}^n r_ia_{ij} <2$ for all $j=1, \ldots, n$. Here $g_0^T$ is the exponential grading operator of $f^T$. Then we define
\[ \rho_1^{r_1} \cdots \rho_n^{r_n} \cdot g_0^T \mapsto x_1^{r_1} \cdots x_n^{r_n}dx_1 \wedge \cdots \wedge dx_n. \]
We show that this mapping induces a one-to-one correspondence between elements of age 1 of $G^T$ which only fix the origin  and elements in $(\Omega_{f,1})^G$ represented by such monomial $n$-forms of weighted degree  equal to 1.

Let $g=\prod_{i=1}^n \rho_i^{r_i} \cdot g_0^T \in {\rm Hom}(G, \CC^\ast)$ be an element of age 1 which only fixes the origin. Then, by Equation~(\ref{Eq:age}), $\sum_{i=1}^n (r_i+1) q_i=1$. This means that the degree of the form $x_1^{r_1} \cdots x_n^{r_n}dx_1 \wedge \cdots \wedge dx_n$ is equal to 1. By definition of $G^T$,
\[ (r_1+1, \ldots, r_n+1) E^{-1} \begin{pmatrix} s_1 \\ \vdots \\ s_n \end{pmatrix} \in \ZZ \mbox{ for all } \prod_{j=1}^n \sigma_j^{s_j} \in G. \]
Therefore the form $x_1^{r_1} \cdots x_n^{r_n}dx_1 \wedge \cdots \wedge dx_n$ is left invariant by $G$.

Conversely, let $x_1^{r_1} \cdots x_n^{r_n}dx_1 \wedge \cdots \wedge dx_n$ be a monomial $n$-form of degree  1 which is left invariant by $G$ and define $u_j:=\sum_{i=1}^n (r_i+1)a_{ij}$ for $j=1, \ldots , n$. Then $\sum_{j=1}^n u_j=1$. 

We claim that $0 < u_i < 1$ for all $i=1, \ldots, n$. We have
\begin{equation}
(r_1+1, \ldots, r_n+1) = (r_1+1, \ldots, r_n+1)E^{-1}E =(u_1, \ldots, u_n)E. \label{Eq:u}
\end{equation}
The polynomial $f(x_1, \ldots , x_n)$ is a (Thom-Sebastiani) sum of invertible polynomials (in groups of different variables) of Fermat, chain, or loop type. It suffices to prove the claim if $f(x_1, \ldots , x_n)$ is of Brieskorn-Pham type (i.e.\ a sum of Fermat monomials) or of chain or loop type. 

Let  $f(x_1, \ldots , x_n)$ be of Brieskorn-Pham type
\[f(x_1, \ldots , x_n) = x_1^{a_1} + \cdots + x_n^{a_n}, \quad n \geq 3. \]
Then it follows from Equation~(\ref{Eq:u}) that $u_i=\frac{r_i+1}{a_i}$.
Therefore the claim follows since $\sum_{i=1}^n u_i=1$  implies that $r_i+1 < a_i$  in view of  $r_i+1 > 0$ for all $i=1, \ldots , n$.

Let $f(x_1, \ldots , x_n)$ be of chain type
\[ f(x_1, \ldots , x_n)=x_1^{a_1}x_2 + x_2^{a_2}x_3 + \ldots + x_{n-1}^{a_{n-1}}x_n + x_n^{a_n}. \]
Then Equation~(\ref{Eq:u}) gives $a_1 u_1=r_1+1$ and $u_{i-1} + a_i u_i=r_i+1$ for $i=2, \ldots , n$. Again it follows that  $0 < r_i+1 < a_i$ for all $i=1, \ldots, n$. Therefore $0 < u_1 <1$ and it follows inductively that $0 < u_i <1$ also for $i=2, \ldots, n$.

Finally, let $f(x_1, \ldots , x_n)$ be of loop type
\[ f(x_1, \ldots , x_n)=x_1^{a_1}x_2 + x_2^{a_2}x_3 + \ldots + x_{n-1}^{a_{n-1}}x_n + x_n^{a_n}x_1.\]
Then Equation~(\ref{Eq:u}) gives $u_n+a_1 u_1=r_1+1$ and $u_{i-1} + a_i u_i=r_i+1$ for $i=2, \ldots , n$. 
Again we have $0 < r_i +1 < a_i$ for all $i=1, \ldots, n$. From this we can again derive that $0< u_i <1$ for all $i=1, \ldots , n$. This proves our claim.

Therefore we can apply  Formula~(\ref{Eq:age}) to the element $g=\prod_{i=1}^n \rho_i^{r_i} \cdot g_0^T$ and see that $g$ is an element of age 1 whose fixed locus is only the origin. This defines a mapping 
\[(\Omega_{f,1})^G \to \{ g \in {\rm Hom}(G, \CC^\ast) \, | \, {\rm age}(g)=1, \ {\rm Fix}\, g = \{ 0 \} \} 
\] 
which is inverse to the above one. This proves the proposition.
\end{proof}

\section{Dolgachev numbers for orbifold curves} \label{sect:Dol}
Let $f(x,y,z)$ be an invertible polynomial  and $G \subset G^{\rm fin}_f$ be a subgroup with $G_0 \subset G$. Let $\widetilde{G}$ be the subgroup of $G_f$ 
defined by the commutative diagram of the short exact sequences
\[ 
\xymatrix{
\{ 1 \}\ar[r] & \CC^\ast \ar[r]\ar@{=}[d] & \widetilde{G} \ar[r]\ar@{^{(}->}[d] & G \ar[r]\ar@{^{(}->}[d] & \{ 1 \}\\
\{ 1 \}\ar[r] & \CC^\ast \ar[r]& G_f \ar[r] & G^{\rm fin}_f \ar[r] & \{ 1 \}
}.
\]
We can naturally associate to the pair $(f,G)$  
the following quotient stack:
\begin{equation}
{\mathcal C}_{(f,G)}:=\left[f^{-1}(0)\backslash\{0\}\left/\widetilde{G} \right.\right]
\end{equation}
Since $f$ has an isolated singularity only at the origin $0\in\CC^3$ 
and $G$ is an extension of a one dimensional torus $\CC^\ast$ 
by a finite abelian group, the stack ${\mathcal C}_{(f,G)}$ is a Deligne--Mumford stack 
and may be regarded as a smooth projective curve 
with a finite number of isotropic points on it. Let $C_{(f,G)} = [{\mathcal C}_{(f,G)}]$ be the underlying smooth projective curve and $g_{(f,G)}$ its genus.

\begin{proposition} \label{prop:gjac}
We have
\[ g_{(f,G)} = \dim_\CC (\Omega_{f,1})^G. \]
\end{proposition}

\begin{proof} We have
\[ \Omega_f = \Omega^3(\CC^3)/(df \wedge \Omega^2(\CC^3)) \simeq {\rm Jac}(f) dx \wedge dy \wedge dz. \]
If $x^{r_1} y^{r_2} z^{r_3}dx \wedge dy \wedge dz$ represents an element of $\Omega_{f,1}$, then the (fractional) degree of the monomial $x^{r_1} y^{r_2} z^{r_3}$ is $1- \sum_{i=1}^3 q_i$. Therefore the proposition follows from \cite[3.4.4, 4.4.5]{Dolgachev82}.
\end{proof}

\begin{definition}
The {\em Dolgachev numbers} $\alpha_1, \ldots, \alpha_r$ of the pair $(f, G)$  are the orders of the isotropy groups of $G$ at the isotropic points of ${\mathcal C}_{(f,G)}$. We denote them by $A_{(f,G)}$.
\end{definition}

For an orbifold ${\mathcal X}$, denote by $H^{p,q}_{\rm st}({\mathcal X})$ ($p,q\in\QQ_{\geq 0}$) 
its $p$-th orbifold Dolbeault cohomology group defined by \cite{CR}. 
Let $h_{\rm st}^{p,q}({\mathcal X})= \dim_\CC H^{p,q}_{\rm st}({\mathcal X})$ be its $(p,q)$-th 
orbifold Hodge number (see also \cite{BD}).

\begin{definition}
The {\em stringy Euler number} of the orbifold curve ${\mathcal C}_{(f,G)}$ is defined by
\[ e_{\rm st}({\mathcal C}_{(f,G)}) := \sum_{p,q\in \QQ_{\geq 0}} (-1)^{p-q} 
h_{\rm st}^{p,q}({\mathcal C}_{(f,G)}). \]
\end{definition}

\begin{proposition} \label{cor:chiDol}
For the stringy Euler number of ${\mathcal C}_{(f,G)}$ we have
\[ e_{\rm st}({\mathcal C}_{(f,G)}) = 2-2g_{(f,G)} +\sum_{i=1}^r (\alpha_i-1). \]
\end{proposition}

\begin{proof} This follows from \cite{CR} since $\alpha_i-1$ is the number of twisted sectors around the corresponding isotropic point (see, in particular, \cite[Example 5.3]{CR}).
\end{proof}

We consider the pair $(f,G_f^{\rm fin})$. By \cite[Theorem~3]{ET},  the quotient stack ${\mathcal C}_{(f,G^{\rm fin}_f)}$ is a smooth projective line $\PP^1$ with at most three isotropic points $p_1, p_2 , p_3$  of orders $\alpha_1', \alpha_2', \alpha_3'$ respectively (we allow that some of the numbers $\alpha_i'$ might be equal to one). The numbers $A_{(f,G_f^{\rm fin})} = (\alpha_1', \alpha_2', \alpha_3')$ are the Dolgachev numbers of the pair $(f,G_f^{\rm fin})$.
For positive integers $u$ and $v$, by $u \ast v$ we denote $v$ copies of the integer $u$.

\begin{theorem} \label{thm:Dol}
Let $G \subset H_i \subset G^{\rm fin}_f$ be the minimal subgroup containing the isotropy group of the point $p_i$, $i=1,2,3$. Then we have the following formula for the Dolgachev numbers $\alpha_1, \ldots, \alpha_r$ of the pair $(f,G)$:
\[ (\alpha_1, \ldots, \alpha_r) = \left( \frac{\alpha'_i}{|H_i/G|} \ast |G^{\rm fin}_f/H_i|, i=1,2,3 \right) ,\]
where we omit numbers which are equal to one on the right-hand side.
\end{theorem}

\begin{proof}
We have the following quotients of stacks
\[
\begin{CD}
{\mathcal C}_{(f,G)} \\
@VV{H_i/G}V \\
{\mathcal C}_{(f,H_i)} \\
@VV{G^{\rm fin}_f/H_i}V \\
{\mathcal C}_{(f,G^{\rm fin}_f)} 
\end{CD}
\]
Since $H_i$ is the minimal subgroup of $G^{\rm fin}_f$ containing the isotropy group of the point $p_i$, there are $|G^{\rm fin}_f/H_i|$ distinct isotropic points of ${\mathcal C}_{(f,H_i)}$ which are identified to the same point $p_i$ under the action of the group $G^{\rm fin}_f/H_i$. Each of them has the order 
$\frac{\alpha_i'}{|H_i/G|}$.
\end{proof}

Since $f$ is weighted homogeneous there are orbit invariants $A^{\CC^\ast}=(\alpha_1^{\CC^\ast}, \ldots , \alpha_s^{\CC^\ast})$ defined by the $\CC^\ast$-action. By \cite[\S 10]{s:2} these numbers can be computed from the weight system as follows. Let  $W_f=(w_1,w_2,w_3;d)$ be the canonical weight system of $f$. Let
\[ 
W_f^{\rm red} = (a_1,a_2,a_3;h):=\left(\frac{w_1}{c_f}, \frac{w_2}{c_f}, \frac{w_3}{c_f}; \frac{d}{c_f} \right)
\]
be the corresponding reduced weight system.
Then the numbers $(\alpha_1^{\CC^\ast}, \ldots , \alpha_s^{\CC^\ast})$ are the following integers: 
\begin{eqnarray*}
a_i & \mbox{if} & a_i \not| \ h, \quad i=1,2,3; \\
{\rm gcd}(a_i,a_j) \ast (m(a_i,a_j;h)-1) & \mbox{if} & {\rm gcd}(a_i,a_j) > 1, \quad 1 \leq i < j \leq 3,
\end{eqnarray*}
where
\[
m(a_i,a_j;h) := |\{ (k,\ell) \in \ZZ^2_{\geq 0} \, | \, ka_i+ \ell a_j =h \}|.
\]

\begin{theorem} \label{thm:strange}
The orbit invariants $(\alpha_1^{\CC^\ast}, \ldots , \alpha_s^{\CC^\ast})$ agree with the Dolgachev numbers of the pair $(f, G_0)$.
\end{theorem}

For the proof, we shall use the classification of invertible polynomials in three variables, as given in Table~\ref{TabAppendix1} (see \cite{KS}, \cite{AGV85}).
\begin{table}[h]
\begin{center}
\begin{tabular}{|c||c||c|}
\hline
{\rm Type} & $f$ & $f^T$ \\
\hline
\hline
I (Fermat) & $x^{p_1}+y^{p_2}+z^{p_3}$ & $x^{p_1}+y^{p_2}+z^{p_3}$\\
&  $(p_1,p_2,p_3\in\ZZ_{\ge 2})$ & $(p_1,p_2,p_3\in\ZZ_{\ge 2})$ \\
\hline
II (Fermat + chain) & $x^{p_1}+y^{p_2}+yz^{\frac{p_3}{p_2}}$ & $x^{p_1}+y^{p_2}z+z^{\frac{p_3}{p_2}}$\\
&  $(p_1,p_2,\frac{p_3}{p_2}\in\ZZ_{\ge 2})$ & $(p_1,p_2,\frac{p_3}{p_2}\in\ZZ_{\ge 2})$\\
\hline
III (Fermat + loop) & $x^{p_1}+zy^{q_2+1}+yz^{q_3+1}$ & $x^{p_1}+zy^{q_2+1}+yz^{q_3+1}$\\
& $(p_1\in\ZZ_{\ge 2}$, $q_2,q_3\in\ZZ_{\ge 1})$ 
& $(p_1\in\ZZ_{\ge 2}$, $q_2,q_3\in\ZZ_{\ge 1})$\\
\hline
IV (chain)& $x^{p_1}+xy^{\frac{p_2}{p_1}}+yz^{\frac{p_3}{p_2}}$ 
& $x^{p_1}y+y^{\frac{p_2}{p_1}}z+z^{\frac{p_3}{p_2}}$\\
&  $(p_1,\frac{p_3}{p_2}\in\ZZ_{\ge 2}, \frac{p_2}{p_1}\in\ZZ_{\ge 1})$
& $(p_1,\frac{p_3}{p_2}\in\ZZ_{\ge 2}, \frac{p_2}{p_1}\in\ZZ_{\ge 1})$\\
\hline
V (loop)& $x^{q_1}y+y^{q_2}z+z^{q_3}x$ & $zx^{q_1}+xy^{q_2}+yz^{q_3}$\\
& $(q_1,q_2,q_3\in\ZZ_{\ge 1})$ & $(q_1,q_2,q_3\in\ZZ_{\ge 1})$  \\
\hline
\end{tabular}
\end{center}
\caption{Invertible polynomials in $3$ variables}\label{TabAppendix1}
\end{table}
For two positive integers $a,b$ we denote their greatest common divisor by $(a,b)$. The following lemma is easy.

\begin{lemma} \label{lem:ka+lb=ab}
Let $a,b$ be positive integers with $c=(a,b)$. Then the equation
\[
k a + \ell b = ab
\]
has $c+1$ solutions $(k,\ell) \in \ZZ_{\geq 0}^2$, namely 
\[
(k, \ell) = \left(i \frac{b}{c}, (c-i) \frac{a}{c}\right), \quad i=0, \ldots , c.
\]
If $b|a$ then there is also a solution $(k, \ell)=(-1,a(b+1)/b)$. If $a|b$ then there is also a solution $(k, \ell)=(b(a+1)/a,-1)$.
\end{lemma}

\begin{proposition} \label{prop:HT}
Let $G \subset H_i \subset G_f^{\rm fin}$ be the minimal subgroup containing the isotropy group of  the $i$-th isotropic point of ${\mathcal C}_{(f,G^{\rm fin}_f)}$, $i=1,2,3$. Then there exists an ordering of the isotropic points such that $H_i^T$ is the maximal subgroup of $G^T$ fixing the $i$-th coordinate. Moreover, 
$G^T/H_i^T \simeq {\rm Hom}(H_i/G, \CC^\ast)$.
\end{proposition}

\begin{proof} By Proposition~\ref{Prop:Gfin/G}, $H_i^T \simeq {\rm Hom}(G^{\rm fin}_f/H_i, \CC^\ast)$. Denote the $i$-th isotropic point of ${\mathcal C}_{(f,G^{\rm fin}_f)}$ by $p_i$. Since each element $(\lambda_1,\lambda_2, \lambda_3)\in G_f^{\rm fin}$ acts on $\CC^3$ diagonally,
the point $p_i$ must be contained in the subvariety $\{xyz=0\}\subset C_{(f,G_f^{\rm fin})}$, hence has one coordinate equal to zero. Since $H_i$ is the minimal subgroup containing the isotropy group of this point, the group $G_f^{\rm fin}/H_i$ is the maximal subgroup fixing this coordinate.

For the proof of the second statement, apply the functor ${\rm Hom}( -, \CC^\ast)$ to the short exact sequence
\[ 
\begin{CD}
\{ 1 \}  \longrightarrow H_i/G \longrightarrow G^{\rm fin}_f/G \longrightarrow G^{\rm fin}_f/H_i\longrightarrow \{ 1 \} .
\end{CD}
\]
Using Proposition~\ref{Prop:Gfin/G}, this yields the short exact sequence
\[ 
\begin{CD}
\{ 1 \}  \longrightarrow H_i^T  \longrightarrow G^T \longrightarrow {\rm Hom}(H_i/G, \CC^\ast)\longrightarrow \{ 1 \} .
\end{CD}
\]
\end{proof}

\begin{proof}[Proof of Theorem \ref{thm:strange}]
The statement will be proved case by case according to the different types of invertible polynomials. In order to compute the Dolgachev numbers of the pair $(f,G_0)$, we use Theorem~\ref{thm:Dol} and the isomorphism ${\rm Hom}(G^{\rm fin}_f/G_0, \CC^\ast) \simeq G_0^T$ given by Proposition~\ref{Prop:Gfin/G}. Let $G_0 \subset H_i \subset G_f^{\rm fin}$ be the minimal subgroup containing the isotropy group of  the $i$-th isotropic point of ${\mathcal C}_{(f,G^{\rm fin}_f)}$, $i=1,2,3$. Denote by $K_j$ the maximal subgroup of $G_0^T$ fixing the coordinate $x_j$, $j=1,2,3$.  We assume that the isotropic points of ${\mathcal C}_{(f,G^{\rm fin}_f)}$ are ordered in such a way that $H_i^T \simeq K_i$ and $G_0^T/K_i \simeq {\rm Hom}(H_i/G_0, \CC^\ast)$. This is possible by Proposition~\ref{prop:HT}.

{\em Type I}: Here $f(x,y,z) =f^T(x,y,z)= x^{p_1}+y^{p_2}+z^{p_3}$. The reduced weight system of $f$ is 
\[
W_f^{\rm red} = \left( \frac{p_2p_3}{c_f}, \frac{p_1p_3}{c_f}, \frac{p_1p_2}{c_f} ; \frac{p_1p_2p_3}{c_f} \right).
\]
All the weights divide the degree of $f$. By Lemma~\ref{lem:ka+lb=ab},
\[
m\left(\frac{p_2p_3}{c_f}, \frac{p_1p_3}{c_f} ; \frac{p_1p_2p_3}{c_f} \right) = (p_1,p_2)+1, \quad
\left(\frac{p_2p_3}{c_f}, \frac{p_1p_3}{c_f} \right) = \frac{p_3}{c_f} \cdot (p_1,p_2).
\]
By symmetry, the same results hold for the other pairs of weights.
Hence
\[
A^{\CC^\ast} = \left( \frac{p_3(p_1,p_2)}{c_f}  \ast (p_1,p_2), \frac{p_2(p_1,p_3)}{c_f}  \ast (p_1,p_3), \frac{p_1(p_2,p_3)}{c_f}  \ast (p_2,p_3) \right).
\]
We have
\[ A_{(f,G_f^{\rm fin})}  =  ( p_1, p_2, p_3). \]
We have $K_3=\langle g_{0,f'} \rangle^T$ where $f'(x,y) = x^{p_1}+y^{p_2}$. According to 
Proposition~\ref{Prop:ordJT}, $K_3$ has order $(p_1,p_2)$. The same arguments can be applied to the other pairs of variables. This shows that $A_{(f, G_0)}=A^{\CC^\ast}$.

{\em Type II}: The equation is $f(x,y,z)=x^{p_1}+y^{p_2}+yz^{\frac{p_3}{p_2}}$ where $p_2 | p_3$. The reduced weight system of $f$ is
\[
W_f^{\rm red} = \left( \frac{p_3}{c_f}, \frac{p_1p_3}{p_2c_f}, \frac{(p_2-1)p_1}{c_f} ; \frac{p_1p_3}{c_f} \right).
\]
We have 
\begin{eqnarray*}
A^{\CC^\ast} & = & \left( \frac{(p_2-1)p_1}{c_f} \left( \mbox{ if } \frac{(p_2-1)p_1}{c_f} \not| \ \frac{p_1p_3}{c_f} \right),  \right. \\
 & & \left(\frac{p_3}{c_f}, \frac{p_1p_3}{p_2c_f} \right) \ast \left( m\left(\frac{p_3}{c_f}, \frac{p_1p_3}{p_2c_f}; \frac{p_1p_3}{c_f} \right) -1 \right) , \\
  & &  \left. \left(\frac{p_1p_3}{p_2c_f},  \frac{(p_2-1)p_1}{c_f} \right) \ast \left( m\left(\frac{p_1p_3}{p_2c_f},  \frac{(p_2-1)p_1}{c_f}; \frac{p_1p_3}{c_f} \right) -1 \right) \right). \\
\end{eqnarray*}
We have 
\[
\left(\frac{p_3}{c_f}, \frac{p_1p_3}{p_2c_f} \right) = \frac{p_3}{p_2c_f} \cdot (p_1.p_2), \quad 
\left(\frac{p_1p_3}{p_2c_f},  \frac{(p_2-1)p_1}{c_f} \right) = \frac{p_1}{c_f} \cdot \left(\frac{p_3}{p_2}, p_2-1\right).
\]
By Lemma~\ref{lem:ka+lb=ab}, we obtain
\[ m\left(\frac{p_3}{c_f}, \frac{p_1p_3}{p_2c_f}; \frac{p_1p_3}{c_f} \right) = (p_1, p_2)+1. \]
Moreover, 
\begin{eqnarray*}
m\left(\frac{p_1p_3}{p_2c_f},  \frac{(p_2-1)p_1}{c_f}; \frac{p_1p_3}{c_f} \right) & = & 
\left| \left\{ (k,\ell) \in \ZZ_{\geq 0}^2 \, \left| \,  k\frac{p_1p_3}{p_2c_f} + \ell  \frac{(p_2-1)p_1}{c_f} = \frac{p_1p_3}{c_f} \right\} \right| \right. \\
 & = & \left| \left\{ (k,\ell) \in \ZZ_{\geq 0}^2 \, \left| \,  (k-1)\frac{p_3}{p_2} + \ell (p_2-1) = \frac{p_3}{p_2}(p_2-1) \right\} \right| \right. .
\end{eqnarray*}
By Lemma~\ref{lem:ka+lb=ab}, the last number is equal to $(\frac{p_3}{p_2}, p_2-1)+1$ if $(p_2-1) \! \not| \, p_3$ and $(\frac{p_3}{p_2}, p_2-1)+2$ if $(p_2-1) | p_3$. Therefore we obtain
\[ A^{\CC^\ast}= \left( \frac{(p_2-1)p_1}{c_f}, \frac{p_3(p_1,p_2)}{p_2c_f}  \ast (p_1,p_2), \frac{p_1\left(\frac{p_3}{p_2}, p_2-1\right)}{c_f}  \ast \left(\frac{p_3}{p_2}, p_2-1\right) \right). \]

On the other hand, we have
\[ A_{(f,G_f^{\rm fin})} = \left( p_1, (p_2-1)p_1, \frac{p_3}{p_2} \right). \]
We show that $K_2$ is trivial. Let $p_2=a c_f +b$ with $a,b \in \ZZ$, $0 \leq b < c_f$. Define
$r_1:= \frac{p_1(b-1)}{c_f}, r_2:= a, r_3:= -\frac{bp_3}{p_2c_f}$.
Then $r_1,r_2,r_3 \in \ZZ$ and $\rho_1^{r_1}\rho_2^{r_2}\rho_3^{r_3}$ is an element of $G_0^T$ which multiplies the coordinate $y$ by $e^{2 \pi \sqrt{-1}/c_f}$. This yields the element $(p_2-1)p_1/c_f \in  A_{(f, G_0)}$.
We have $K_3=\langle g_{0,f'} \rangle^T$ where $f'(x,y) = x^{p_1}+y^{p_2}$. According to 
Proposition~\ref{Prop:ordJT}, $K_3$ has order $(p_1,p_2)$. Moreover,  $K_1=\langle g_{0,f''} \rangle^T$ where $f''(y,z)=y^{p_2}+yz^{\frac{p_3}{p_2}}$ and $|K_1|=(\frac{p_3}{p_2}, p_2-1)$. This shows that $A_{(f, G_0)}=A^{\CC^\ast}$.

{\em Type III}: Here $f(x,y,z)=f^T(x,y,z) = x^{p_1}+zy^{q_2+1}+yz^{q_3+1}$.
The reduced weight system of $f$ is
\[ W_f^{\rm red} = \left( \frac{p_2}{c_f}, \frac{p_1q_3}{c_f}, \frac{p_1q_2}{c_f}; \frac{p_1p_2}{c_f} \right) \]
where $p_2+1=(q_2+1)(q_3+1)$. We have
\[ \left( \frac{p_2}{c_f}, \frac{p_1q_3}{c_f} \right) = \left( \frac{p_2}{c_f}, \frac{p_1q_2}{c_f} \right)=1\]
since a common divisor of $p_2$ and of one of the numbers $q_2$ and $q_3$ has to divide the other one as well because $p_2=q_2(q_3+1)+q_3=q_3(q_2+1)+q_2$. Hence it has to divide $c_f$. Therefore we have 
\begin{eqnarray*}
A^{\CC^\ast} & = & \left( \frac{p_1q_3}{c_f} \left( \mbox{ if } \frac{p_1q_3}{c_f} \not| \ \frac{p_1p_2}{c_f} \right),  
 \frac{p_1q_2}{c_f} \left( \mbox{ if } \frac{p_1q_2}{c_f} \not| \ \frac{p_1p_2}{c_f} \right),  \right.  \\
  & & \left.  \left(\frac{p_1q_2}{c_f},  \frac{p_1q_3}{c_f} \right) \ast \left( m\left(\frac{p_1q_2}{c_f},  \frac{p_1q_3}{c_f}; \frac{p_1p_2}{c_f} \right) -1 \right) \right). \\
\end{eqnarray*}
We have
\begin{eqnarray*}
\left(\frac{p_1q_2}{c_f},  \frac{p_1q_3}{c_f} \right)  & = & \frac{p_1}{c_f} \cdot (q_2,q_3), \\
m\left(\frac{p_1q_2}{c_f},  \frac{p_1q_3}{c_f}; \frac{p_1p_2}{c_f} \right) & = & 
\left| \left\{ (k,\ell) \in \ZZ_{\geq 0}^2 \, \left| \,  k\frac{p_1q_2}{c_f} + \ell \frac{p_1q_3}{c_f}= \frac{p_1p_2}{c_f} \right\} \right| \right. \\
 & = & \left| \left\{ (k,\ell) \in \ZZ_{\geq 0}^2 \, \left| \,  (k-1)q_2 + (\ell-1)q_3 = q_2q_3 \right\} \right| \right. .
\end{eqnarray*}
By Lemma~\ref{lem:ka+lb=ab}, the last number is equal to $(q_2,q_3)+1$ if $q_2 \! \not| \, q_3$ and 
$q_3 \! \not| \, q_2$, to $(q_2,q_3)+2$ if $q_2|q_3$ or $q_3|q_2$ but not $q_2=q_3$, and to $(q_2,q_3)+3$ if $q_2=q_3$.
Hence we obtain
\[ A^{\CC^\ast} = \left( \frac{p_1q_2}{c_f}, \frac{p_1q_3}{c_f}, \frac{p_1(q_2,q_3)}{c_f} \ast (q_2,q_3) \right). \]

On the other hand, we have
\[ A_{(f,G_f^{\rm fin})} = (p_1,p_1q_3,p_1q_2). \]
We show that $K_2$ is trivial. Let $r_1:= \frac{p_1q_2}{c_f}, r_2:=1, r_3:=q_3+1- \frac{p_2}{c_f}$. Then $r_1,r_2,r_3$ are integers and the element $\rho_1^{r_1}\rho_2^{r_2}\rho_3^{r_3} \in G_0^T$ multiplies the coordinate $y$ by $e^{2 \pi \sqrt{-1}/c_f}$. So we see that $\frac{p_1q_3}{c_f} \in A_{(f, G_0)}$. For symmetry reasons, also $K_3= \{ 1 \}$ and $\frac{p_1q_2}{c_f} \in A_{(f, G_0)}$. Finally,  $K_1=\langle g_{0,f'} \rangle^T$ where $f'(y,z)=zy^{q_2+1}+yz^{q_3+1}$ and $|K_1|=(q_2,q_3)$. Hence the numbers $\frac{p_1(q_2,q_3)}{c_f} \ast (q_2,q_3)$ are in $A_{(f, G_0)}$.

{\em Type IV}: Here $f(x,y,z)=x^{p_1}+xy^{\frac{p_2}{p_1}}+yz^{\frac{p_3}{p_2}}$ where $p_1|p_2$ and $p_2|p_3$. The reduced weight system is
\[ W_f^{\rm red} = \left( \frac{p_3}{p_1c_f}, \frac{(p_1-1)p_3}{p_2c_f}, \frac{p_2-p_1+1}{c_f}; \frac{p_3}{c_f} \right). \]
The orbit invariants are 
\begin{eqnarray*}
A^{\CC^\ast} & = & \left( \frac{p_2-p_1+1}{c_f},  
 \frac{(p_1-1)p_3}{p_2c_f} \left( \mbox{ if } \frac{(p_1-1)p_3}{p_2c_f} \not| \ \frac{p_3}{c_f} \right),  \right.  \\
  & & \left.  \left( \frac{p_3}{p_1c_f}, \frac{(p_1-1)p_3}{p_2c_f}, \right) \ast \left( m\left(\frac{p_3}{p_1c_f}, \frac{(p_1-1)p_3}{p_2c_f}; \frac{p_3}{c_f} \right) -1 \right) \right). \\
\end{eqnarray*}
A similar reasoning as for Type III shows
\[ A^{\CC^\ast}= \left(  \frac{p_2-p_1+1}{c_f},  \frac{(p_1-1)p_3}{p_2c_f}, \frac{p_3\left(\frac{p_2}{p_1},p_1-1 \right)}{p_2c_f} \ast \left(\frac{p_2}{p_1},p_1-1 \right) \right). \]

On the other hand, we have
\[ A_{(f,G_f^{\rm fin})} = \left( p_2-p_1+1, (p_1-1) \frac{p_3}{p_2}, \frac{p_3}{p_2} \right). \]
The group $K_1$ is trivial, since, by Proposition~\ref{Prop:genJT} (i), the element $\rho_n^{d/c_f} \in G_0^T$ multiplies the coordinate $x$ by $e^{2 \pi \sqrt{-1}/c_f}$. Hence $\frac{p_2-p_1+1}{c_f} \in A_{(f, G_0)}$. Moreover, the group $K_2$ is trivial: Let $\frac{p_2}{p_1} = ac_f +b$ with $a,b \in \ZZ$, $0 \leq b < c_f$. Define
$r_1:= p_1(1-a) + \frac{p_2-p_1+1}{c_f},  r_2:= a,  r_3:= - \frac{bp_3}{p_2c_f}$.
Then $r_1, r_2, r_3$ are integers and the element $\rho_1^{r_1}\rho_2^{r_2}\rho_3^{r_3} \in G_0^T$ multiplies the coordinate $y$ by $e^{2 \pi \sqrt{-1}/c_f}$. Hence $\frac{(p_1-1)p_3}{p_2c_f} \in A_{(f, G_0)}$. Finally,  $K_3=\langle g_{0,f'} \rangle^T$ where $f'(x,y)=x^{p_1}+xy^{\frac{p_2}{p_1}}$ and $|K_3|=\left(\frac{p_2}{p_1},p_1-1 \right)$. This shows that $A_{(f, G_0)}=A^{\CC^\ast}$.

{\em Type V}: Here $f(x,y,z)=f^T(x,y,z) = x^{q_1}y+y^{q_2}z+z^{q_3}x$. The reduced weight system is 
\[ W_f^{\rm red} = \left( \frac{q_2q_3-q_3+1}{c_f}, \frac{q_3q_1-q_1+1}{c_f}, \frac{q_1q_2-q_2+1}{c_f}; \frac{q_1q_2q_3+1}{c_f} \right). \]
One can easily show that the weights of the reduced weight system are coprime and that none of the weights divides the degree. Hence the orbit invariants are
\[ A^{\CC^\ast}= \left( \frac{q_2q_3-q_3+1}{c_f}, \frac{q_3q_1-q_1+1}{c_f}, \frac{q_1q_2-q_2+1}{c_f} \right). \]
The Dolgachev numbers of the pair $(f,G_f^{\rm fin})$ are
\[ A_{(f,G_f^{\rm fin})} = ( q_2q_3-q_3+1, q_3q_1-q_1+1, q_1q_2-q_2+1). \]
All the groups $K_i$, $i=1,2,3$, are trivial, since, by Proposition~\ref{Prop:genJT} (ii), for each coordinate there is an element in $G_0^T$ which multiplies this coordinate by $e^{2 \pi \sqrt{-1}/c_f}$. This proves the claim for Type V and finishes the proof of Theorem~\ref{thm:strange}.
\end{proof}

\section{Orbifold E-function and characteristic polynomial} \label{sect:E-function}
Let $n$ be a positive integer.
Consider a polynomial $f\in\CC[x_1,\dots, x_n]$ satisfying 
\[
\dim_\CC\CC[x_1,\dots, x_n]\left/\left(\frac{\partial f}{\partial x_1},\dots,\frac{\partial f}{\partial x_n}\right)<\infty\right. .
\] 
We regard the polynomial $f$ as a holomorphic map $f:X\to\CC$ where 
$X$ is a suitably chosen neighborhood of $0\in\CC^n$ so that the fibration 
$f$ has good technical properties.
Consider the {\it Milnor fiber} $X_f:=\{x\in X~|~f(x)=1\}$ of the fibration $f:X\to\CC$. 
It is well-known that $H^{n-1}(X_f,\CC)$ can be equipped with a canonical mixed Hodge structure 
with an automorphism, which is constructed by Steenbrink \cite{st:1} and the automorphism denoted by $c$ 
is given by the Milnor monodromy. 
We can naturally associate a bi-graded vector space to a mixed Hodge structure 
with an automorphism.
Consider the Jordan decomposition $c=c_{\rm ss}\cdot c_{\rm unip}$ of $c$ where $c_{\rm ss}$ and $c_{\rm unip}$ denote 
the semi-simple part and unipotent part respectively.
For $\lambda\in \CC$, let  
\begin{equation}
H^{n-1}(X_f,\CC)_\lambda:={\rm Ker}(c_{\rm ss}-\lambda\cdot {\rm id}:H^{n-1}(X_f,\CC)\longrightarrow 
H^{n-1}(X_f,\CC)).
\end{equation}
Denote by $F^\bullet$ be the Hodge filtration of the mixed Hodge structure.
\begin{definition}
Define the bi-graded vector space $\H^{p,q}_f$ indexed by $\QQ$ as follows:
\begin{enumerate}
\item If $p+q\ne n$, then  $\H^{p,q}_f:=0$.
\item If $p+q=n$ and $p\in\ZZ$, then  
\begin{equation}
\H^{p,q}_f:={\rm Gr}^{p}_{F^\bullet}H^{n-1}(X_f,\CC)_1.
\end{equation}
\item If $p+q=n$ and $p\notin\ZZ$, then  
\begin{equation}
\H^{p,q}_f:={\rm Gr}^{[p]}_{F^\bullet}H^{n-1}(X_f,\CC)_{e^{-2\pi\sqrt{-1} p}},
\end{equation}
where $[p]$ is the largest integer less than $p$.
\end{enumerate}
\end{definition}
Clearly, we have an isomorphism 
$H^{n-1}(X_f,\CC)\simeq \oplus_{p,q\in\QQ}\H^{p,q}_f$ of vector spaces.
\begin{remark}
The rational number $q$ with $\H^{p,q}_f\ne 0$ is called the {\em exponent} of 
the singularity $f$. 
The set of exponents is called the {\em spectrum} of the singularity $f$, 
which is one of the most important invariants of $f$. 
\end{remark}
It is very useful to introduce the generating function of exponents:
\begin{definition}
The {\em E-function} of $f$ is  (cf.\ \cite{BD})
\begin{equation}
E(f;t,\bar{t}):=\sum_{p,q\in\QQ}(-1)^{n}\dim_\CC \H^{p,q}_f\cdot 
t^{p-\frac{n}{2}}\bar{t}^{q-\frac{n}{2}}.
\end{equation}
The {\em characteristic polynomial} of $f$ is 
\begin{equation}
\phi(f;t):= \prod_{q\in\QQ} (t- {\bf e}[q])^{\dim_\CC\H^{p,q}_f}.
\end{equation}
This is a polynomial of degree $\mu_f$ where $\mu_f$ is the {\em Milnor number} of $f$, i.e.\ the dimension of $H^{n-1}(X_f,\CC)$.
\end{definition}
It is easy to see that $E(f;t^{-1},\bar{t}^{-1})=E(f;t,\bar{t})$.

Let us give examples of the E-function. 
By \cite{st:2}, we have the following
\begin{proposition} \label{prop:Eqh}
Let  $f\in\CC[x_1,\dots, x_n]$ be a weighted homogeneous polynomial with 
a system of weights $(w_1, . . . ,w_n;d)$. Set $q_i:=w_i/d$ for $i=1,\dots, n$.
We have 
\begin{equation}
E(f;t,\bar{t})=(-1)^n\prod_{i=1}^n 
\frac{1-\left(\frac{\bar{t}}{t}\right)^{1-q_i}}
{1-\left(\frac{\bar{t}}{t}\right)^{q_i}}\left(\frac{\bar{t}}{t}\right)^{q_i-\frac{1}{2}}.
\end{equation}
\qed
\end{proposition}

Also the E-function of a singularity of type $T_{p_1,p_2,p_3}$ will play a role.
\begin{example}Let $f(x_1,x_2,x_3)= x_1^{p_1} + x_2^{p_2} + x_3^{p_3} - x_1x_2x_3$.
By \cite[(4.2) Example]{st:1}, we have 
\begin{eqnarray}
E(f;t,\bar{t}) & = & -\left( t^{2-\frac{3}{2}}\bar{t}^{1-\frac{3}{2}} +  t^{1-\frac{3}{2}}\bar{t}^{2-\frac{3}{2}}+ \left(\frac{\bar{t}}{t}\right)^{-\frac{3}{2}}\sum_{i=1}^3 
\frac{1-\left(\frac{\bar{t}}{t}\right)^{1-\frac{1}{p_i}}}
{1-\left(\frac{\bar{t}}{t}\right)^{\frac{1}{p_i}}} \right), \\
\phi(f;t) & = & (t-1)^2 \prod_{i=1}^3 \frac{t^{p_i}-1}{t-1}.
\end{eqnarray}
\end{example}

Let $G$ be a finite abelian subgroup of ${\rm GL}_n(\CC)$ acting diagonally on 
the coordinates $(x_1,\dots ,x_n)$ such that $f(x_1,\dots,x_n)$ is invariant under $G$. 
To each pair $(f,G)$ we can associate a natural mixed Hodge structure with an automorphism, 
which gives the following bi-graded vector space:
\begin{definition}
Define the bi-graded $\CC$-vector space $\H_{f,G}$ as 
\begin{equation}
\H_{f,G}:=\bigoplus_{g\in G}(\H_{f^g})^G(-{\rm age}(g),-{\rm age}(g))
\end{equation}
where we set $\CC^{n_g}:=\{x\in\CC^n~|~g\cdot x=x \}$ and $f^g:=f|_{\CC^{n_g}}$.
\end{definition}

\begin{theorem} \label{thm:EGqh}
Let $f$ be a weighted homogeneous polynomial as in Proposition~\ref{prop:Eqh}. Write $g,h \in G$ in the form 
${\rm diag}({\bf e}[q_1a_1], {\bf e}[q_2a_2], {\bf e}[q_3a_3])$,
${\rm diag}({\bf e}[q_1b_1], {\bf e}[q_2b_2], {\bf e}[q_3b_3])$,
for $a_i,b_i\in \ZZ$, $i=1,2,3$. 
The E-function for the pair $(f,G)$ is naturally given by the following formula (cf.\ \cite{t:3})$:$ 
\begin{equation}
E(f,G)(t,\bar{t})=\sum_{g\in G}E_g(f,G)
(t,\bar{t}),
\end{equation}
\begin{equation}
\begin{split}
& E_g(f,G)(t,\bar{t})\\
:=(-1)^{n_g}\left(\frac{\bar{t}}{t}\right)^{-\frac{3}{2}} \cdot \left(\prod_{q_ia_i\not\in\ZZ}\left({t}{\bar{t}}\right)^{\frac{1}{2}} \right)
\cdot \frac{1}{|G|}& \sum_{h\in G}\prod_{q_ia_i\in\ZZ}
\frac{1-{\bf e}\left[(1-q_i)b_i\right]\left(\frac{\bar{t}}{t}\right)^{1-q_i}}
{1-{\bf e}\left[q_ib_i\right]\left(\frac{\bar{t}}{t}\right)^{q_i}}{\bf e}\left[q_ib_i\right]
\left(\frac{\bar{t}}{t}\right)^{q_i}.
\end{split}
\end{equation}
\qed
\end{theorem}

\begin{example}
Let $f(x_1,x_2,x_3)= x_1^{p_1} + x_2^{p_2} + x_3^{p_3} - x_1x_2x_3$ and $G$ be a subgroup of 
${\rm SL}_n(\CC) \cap G_f^{\rm fin}$. Write $g,h \in G$ in the form 
${\rm diag}({\bf e}[\frac{a_1}{p_1}], {\bf e}[\frac{a_2}{p_2}], {\bf e}[\frac{a_3}{p_3}])$,
${\rm diag}({\bf e}[\frac{b_1}{p_1}], {\bf e}[\frac{b_2}{p_2}], {\bf e}[\frac{b_3}{p_3}])$,
for $0\le a_i,b_i< p_i$, $i=1,2,3$.
The E-function for the pair $(f,G)$ is given by the following formula:
\begin{equation}
E(f,G)(t,\bar{t})  =  \sum_{g\in G}E_g(f,G)
(t,\bar{t}),
\end{equation}
where 
\begin{equation}
E_g(f,G)(t,\bar{t}) = \left\{ \begin{array}{cl} E'_{g}(f,G)(t,\bar{t}) -t^{2-\frac{3}{2}}\bar{t}^{1-\frac{3}{2}} -  t^{1-\frac{3}{2}}\bar{t}^{2-\frac{3}{2}} & \mbox{if } g={\rm id}_G, \\
E'_g(f,G)(t,\bar{t}) & \mbox{otherwise}, \end{array} \right.
\end{equation}
and
\begin{equation}
\begin{split}
& E'_g(f,G)(t,\bar{t})\\
 & := (-1)^{n_g} 
  \left(\frac{\bar{t}}{t}\right)^{-\frac{3}{2}} \cdot
\left(\prod_{i:\frac{a_i}{p_i}\not\in\ZZ}\left({t}{\bar{t}}\right)^{\frac{1}{2}}\right)\cdot
\frac{1}{|G|}\sum_{h\in G}\left(\sum_{i:\frac{a_i}{p_i}\in\ZZ}
\frac{1-{\bf e}\left[(1-\frac{1}{p_i})b_i \right]\left(\frac{\bar{t}}{t}\right)^{1-\frac{1}{p_i}}}
{1-{\bf e}\left[\frac{b_i}{p_i}\right] \left(\frac{\bar{t}}{t}\right)^{\frac{1}{p_i}}} {\bf e}\left[\frac{b_i}{p_i}\right] \left(\frac{\bar{t}}{t}\right)^{\frac{1}{p_i}}\right).
\end{split}
\end{equation}
\end{example}

\begin{definition}
Let $f(x_1, \ldots, x_n)$ be a polynomial which is invariant under a finite subgroup $G \subset {\rm GL}_n(\CC)$. 
The {\em characteristic polynomial} for the pair $(f,G)$ is defined by 
\begin{equation}
\phi(f,G)(t) := \prod_{g \in G} \phi_g(f,G)(t)^{(-1)^{n_g-1}}
\end{equation}
where $\phi_g(f,G)(t)$ is the characteristic polynomial of $f^g$ obtained from 
$E_g(f,G)(t, \bar{t})$. 

The {\em  Milnor number} $\mu_{(f,G)}$ of the pair $(f,G)$ is defined to be (cf.\ \cite{Wa})
\[ \mu_{(f,G)} := \sum_{g \in G} (-1)^{n_g-1} \mu_{f^{g}}. \]
\end{definition}

\begin{theorem} \label{thm:cusp}
Let $f(x_1,x_2,x_3)= x_1^{p_1} + x_2^{p_2} + x_3^{p_3} - x_1x_2x_3$ and $G$ be a subgroup of ${\rm SL}_3(\CC) \cap G_f^{\rm fin}$. Let $H_i \subset G$ be the maximal subgroup fixing the coordinate $x_i$, $i=1,2,3$. Define numbers $\gamma_1, \ldots , \gamma_s$ by
\[ (\gamma_1, \ldots, \gamma_s) = \left( \frac{p_i}{|G/H_i|} \ast |H_i|, i=1,2,3 \right) ,\]
where we omit numbers which are equal to one on the right-hand side. Then
\[ \phi(f,G)(t) = (t-1)^{2-2j_G} \prod_{i=1}^s \frac{t^{\gamma_i}-1}{t-1}. \]
\end{theorem}

\begin{proof} Since $G \subset {\rm SL}_3(\CC)$, each element $g \in G$, $g \neq {\rm id}_G$, has a 1-dimensional or 0-dimensional fixed locus. The elements $g \in G$ which fix the coordinate $x_i$ form a cyclic group $H_i$ of order $|H_i|$. Let $g \in H_i$, $g \neq {\rm id}_G$. For simplicity, we assume that $i=1$. Then $g=(0,{\bf e}[\frac{a_2}{p_2}], {\bf e}[\frac{a_3}{p_3}])$. We have
\begin{eqnarray*}
E_g(f,G)(t,\bar{t}) & = &
 - t\bar{t} \left(\frac{\bar{t}}{t}\right)^{-\frac{3}{2}}  
 \frac{1}{|G|} \sum_{h\in G}
\frac{1-{\bf e}\left[(1-\frac{1}{p_1})b_1 \right]\left(\frac{\bar{t}}{t}\right)^{1-\frac{1}{p_1}}}
{1-{\bf e}\left[\frac{b_1}{p_1}\right]\left(\frac{\bar{t}}{t}\right)^{\frac{1}{p_1}}} 
{\bf e}\left[\frac{b_1}{p_1}\right] \left(\frac{\bar{t}}{t}\right)^{\frac{1}{p_1}} \\
& = &  -  t\bar{t} \left(\frac{\bar{t}}{t}\right)^{-\frac{3}{2}} 
 \frac{1}{|G|}\sum_{h\in G} \left( 
 \sum_{k=0}^{p_1-2} {\bf e}\left[k\frac{b_1}{p_1}\right] \left(\frac{\bar{t}}{t}\right)^{\frac{k}{p_1}} \right) 
 {\bf e}\left[\frac{b_1}{p_1}\right] \left(\frac{\bar{t}}{t}\right)^{\frac{1}{p_1}} \\
& = & -  t\bar{t} \left(\frac{\bar{t}}{t}\right)^{-\frac{3}{2}}
\frac{1}{|G|}\sum_{h\in G} \left(
 \sum_{k=0}^{p_1-2} {\bf e}\left[(k+1)\frac{b_1}{p_1}\right] \left(\frac{\bar{t}}{t}\right)^{\frac{k+1}{p_1}} \right).
\end{eqnarray*}
Since $G$ is a finite group, we have for any character $\chi: G \to \CC^\ast$
\begin{equation}
\sum_{h \in G} \chi(h) = \left\{ \begin{array}{cl} |G| & \mbox{if } \chi(h)=1 \mbox{ for all } h \in G,\\ 0 & \mbox{otherwise}. \end{array} \right.
\end{equation}
Therefore
\begin{eqnarray*}
E_g(f,G)(t,\bar{t}) & = &
 t\bar{t} \left(\frac{\bar{t}}{t}\right)^{-\frac{3}{2}}  
\sum_{\ell} \left(\frac{\bar{t}}{t}\right)^{\frac{\ell}{p_1}},
\end{eqnarray*}
where the sum runs over all $\ell$ with $1 \leq \ell \leq p_1-1$ and $\ell \frac{b_1}{p_1} \in \ZZ$ for all $h \in G$. Let $\widetilde{\gamma}_i := \frac{p_i}{|G/H_i|}$, $i=1,2,3$. 
Since $\frac{b_1}{p_1}=0$ for $h \in H_1$ and $G/H_1$ leaves the monomial $x^{p_1}$ invariant, 
we obtain
\begin{eqnarray*}
E_g(f,G)(t,\bar{t}) & = &
 t\bar{t} \left(\frac{\bar{t}}{t}\right)^{-\frac{3}{2}}  
\sum_{\ell=1}^{\widetilde{\gamma}_1-1}\left(\frac{\bar{t}}{t}\right)^{\frac{\ell}{\widetilde{\gamma}_1}}.
\end{eqnarray*}
This implies that the exponents of $f^g$ are
$\frac{1}{\widetilde{\gamma}_1}, \ldots, \frac{\widetilde{\gamma}_1-1}{\widetilde{\gamma}_1}$
and the characteristic polynomial is 
\[ \phi_g(f,G)(t) = \frac{t^{\widetilde{\gamma}_1}-1}{t-1}. \]
There are $|H_1|-1$ elements $g \in H_1$, $g \neq {\rm id}_G$, with this characteristic polynomial.

By a similar reasoning we obtain 
\[  \phi_{{\rm id}_G}(f,G) = (t-1)^2 \prod_{i=1}^3 \frac{t^{\widetilde{\gamma}_i}-1}{t-1}. \]
If $g \in G$ is an element with fixed locus $\{ 0\}$, then we obtain $\phi_g(f,G)(t)= (t-1)$. 
By Remark~\ref{rem:fix}, there are $2j_G$ elements in $G$ with fixed locus $\{ 0\}$. 
Therefore, we finally get
\[ \phi(f,G)(t) = \prod_{g \in G} \phi_g(f,G)(t)^{(-1)^{n_g-1}} = (t-1)^{2-2j_G} \prod_{i=1}^s \frac{t^{\gamma_i}-1}{t-1}. \]
\end{proof}

From Theorem~\ref{thm:cusp} we immediately obtain the following corollary:
\begin{corollary} \label{cor:mucusp}
Let $(f,G)$ be as in Theorem~\ref{thm:cusp}. Then the Milnor number of the pair $(f,G)$ is given by
\[ \mu_{(f,G)} = 2- 2j_G + \sum_{i=1}^s (\gamma_i -1). \]
\end{corollary}

\section{Gabrielov numbers for pairs $(f^T,G^T)$} \label{sect:Gab}
Let $f(x,y,z)$ be an invertible polynomial and let $G \subset {\rm SL}_3(\CC) \cap G^{\rm fin}_f$. Let $(\gamma'_1,\gamma'_2,\gamma'_3):=(\gamma_1^{(f, \{ 1 \})}, \gamma_2^{(f, \{ 1 \})}, \gamma_3^{(f, \{ 1 \})})$ be the Gabrielov numbers of the pair $(f, \{ 1 \})$ defined in \cite{ET}.   Let $\Delta(\gamma'_1,\gamma'_2,\gamma'_3):= \gamma'_1\gamma'_2\gamma'_3- \gamma'_2 \gamma'_3-\gamma'_1 \gamma'_3 - \gamma'_1 \gamma'_2$. If $\Delta(\gamma'_1,\gamma'_2,\gamma'_3) \geq 0$, then there is holomorphic coordinate change of the polynomial $f(x,y,z) +axyz$ for some $a \in \CC^\ast$ ($a=-1$ if $\Delta(\gamma'_1,\gamma'_2,\gamma'_3) > 0$ )  to the polynomial $F(x,y,z):=x^{\gamma'_1} + y^{\gamma'_2} + z^{\gamma'_3} -xyz$ of type $T_{\gamma'_1,\gamma'_2,\gamma'_3}$ (see \cite[Theorem~10(ii),(iii)]{ET}). Since $G \subset {\rm SL}_3(\CC)$, the polynomial $f(x,y,z) +axyz$ is also left invariant by $G$.

\begin{proposition}
The coordinate change from the polynomial $f(x,y,z)+axyz$ to 
the polynomial $F(x,y,z)=x^{\gamma'_1} + y^{\gamma'_2} + z^{\gamma'_3} -xyz$  is $G$-equivariant. In particular, the polynomial $F(x,y,z)$ is also left invariant by $G$.
\end{proposition}

\begin{proof} For simplicity, we assume $a=-1$ and we prove the statement for an invertible polynomial of Type II. The proof in the other cases is similar.
For Type II we have $f(x,y,z)-xyz=x^{p_1}+y^{p_2}+yz^{\frac{p_3}{p_2}}-xyz$ and $F(x,y,z)=x^{p_1}+y^{p_2}+z^{p_1(\frac{p_3}{p_2}-1)}-xyz$ (see \cite{ET}). We have to do coordinate changes to get rid of the mixed term. As the first and main step, we use the coordinate change $(x,y,z) \mapsto (x+z^{\frac{p_3}{p_2}-1},y,z)$. 
We show that this transformation is $G$-equivariant.
Let $g=\sigma_1^{s_1} \sigma_2^{s_2} \sigma_3^{s_3} \in G$. Since $g \in {\rm SL}_3(\CC)$, we have
\begin{equation}
   (1,1,1) E^{-1} (s_1, s_2, s_3)^T =  \sum_{i=1}^3 s_i q_i^T \in \ZZ, \label{eq:111}
\end{equation}
where $q_i^T= \frac{w_i^T}{d}$ ($i=1,2,3$) for the canonical system of weights $W_{f^T}=(w_1^T,w_2^T,w_3^T;d)$ associated to $f^T$. 
Since $f$ is invariant under $G$, we have 
\begin{equation}
    (0,1,\textstyle{\frac{p_3}{p_2}}) E^{-1} (s_1, s_2, s_3)^T  \in \ZZ. \label{eq:a1}
\end{equation}
Subtracting (\ref{eq:111}) from (\ref{eq:a1}) yields
\begin{equation}
 (-1,0,\textstyle{\frac{p_3}{p_2}}-1) E^{-1} (s_1, s_2, s_3)^T  \in \ZZ. 
\end{equation}
This shows that this coordinate change is $G$-equivariant. This coordinate change introduces new mixed terms of higher degree. By applying further coordinate changes of the same type we can push the degree of the mixed terms arbitrarily high, see the proof of \cite[Theorem~10(iii)]{ET}.
\end{proof}

\begin{definition}
The {\em Gabrielov numbers} of the pair $(f,G)$ are the numbers $(\gamma_1, \ldots, \gamma_s)$ for the pair $(F,G)$ defined in Theorem~\ref{thm:cusp}.
We denote the Gabrielov numbers of the pair $(f,G)$ by $\Gamma_{(f,G)}$.
\end{definition}

\section{Mirror symmetry} \label{sect:mirror}
Let $f(x,y,z)$ be an invertible polynomial and $G \subset G^{\rm fin}_f$ be a subgroup with $G_0 \subset G \subset G^{\rm fin}_f$. Then $(G^{\rm fin}_f)^T=\{ 1\} \subset G^T \subset G_0^T = {\rm SL}_3(\CC) \cap G^{\rm fin}_{f^T}$. In Section~\ref{sect:Dol} we defined Dolgachev numbers $A_{(f,G)}$ and the genus $g_{(f,G)}$ for the pair $(f,G)$. By Section~\ref{sect:Gab}, there are Gabrielov numbers $\Gamma_{(f^T,G^T)}$ defined for the pair $(f^T,G^T)$. The number $j_{G^T}$ was defined in Section~\ref{sect:g0}.

\begin{theorem} \label{thm:mirror}
We have
\[ g_{(f,G)} = j_{G^T}, \quad A_{(f,G)} = \Gamma_{(f^T,G^T)}, \quad 
e_{\rm st}({{\mathcal C}_{(f,G)}})=\mu_{(F,G^T)}, \]
where $F$ is a polynomial of type $T_{\gamma'_1,\gamma'_2,\gamma'_3}$ with 
the Gabrielov numbers $(\gamma'_1,\gamma'_2,\gamma'_3)$ for $f^T$.
\end{theorem}

\begin{proof} The first equality $g_{(f,G)} = j_{G^T}$ follows from Proposition~\ref{prop:gjac} and Proposition~\ref{Prop:g=j}. 

\begin{sloppypar}

We now prove that the Dolgachev numbers of the pair $(f,G)$ coincide with the Gabrielov numbers for the pair $(f^T,G^T)$. By \cite{ET},  the Dolgachev numbers $(\alpha_1^{(f,G_f^{\rm fin})}, \alpha_2^{(f,G_f^{\rm fin})}, \alpha_3^{(f,G_f^{\rm fin})})$ for the pair $(f,G_f^{\rm fin})$ coincide with the Gabrielov numbers $(\gamma_1^{(f^T, \{ 1 \})}, \gamma_2^{(f^T, \{ 1 \})}, \gamma_3^{(f^T, \{ 1 \})})$ of the pair $(f^T, \{ 1 \})$. By Proposition~\ref{Prop:Gfin/G} we have natural group isomorphisms
\[ G^{\rm fin}_f/G \simeq {\rm Hom}(G^{\rm fin}_f/G, \CC^\ast) \simeq G^T. \]
By Proposition~\ref{prop:HT}, there is an ordering of the isotropic points of ${\mathcal C}_{(f,G)}$ such that, if $G \subset H_i \subset G_f^{\rm fin}$ is the minimal subgroup containing the isotropy group of the $i$-th point and $K_i \subset G_0^T$ is the maximal subgroup fixing the $i$-th coordinate, we have $H_i^T \simeq K_i$ and $G^T/H_i^T \simeq {\rm Hom}(H_i/G, \CC^\ast)$, $i=1,2,3$.
Therefore the claim follows from Theorem~\ref{thm:Dol} and Theorem~\ref{thm:cusp}.

\end{sloppypar}

Finally, the equality $e_{\rm st}({{\mathcal C}_{(f,G)}})=\mu_{(F,G^T)}$ follows from the previous statements, Proposition~\ref{cor:chiDol}, and Corollary~\ref{cor:mucusp}.
\end{proof}

Let $(w_1,w_2,w_3;d)$ be the canonical system of weights attached to $f$ and let $c=|G^{\rm fin}_f/G|$. 
The ring $R_{(f,G)}:=\CC[x,y,z]/(f)$ is a $\ZZ$-graded ring with respect to the system of 
weights $(\frac{w_1}{c},\frac{w_2}{c},\frac{w_3}{c};\frac{d}{c})$.
Therefore, we can consider the decomposition of $R_{(f,G)}$ 
as a $\ZZ$-graded $\CC$-vector space:
\[
R_{(f,G)}:=\bigoplus_{k\in\ZZ_{\ge 0}} R_{(f,G),k}, \quad R_{(f,G),k}:=\left\{g\in R_{(f,G)}~\left|~
\frac{w_1}{c}x\frac{\partial g}{\partial x}+\frac{w_2}{c}y\frac{\partial g}{\partial y}
+\frac{w_3}{c}z\frac{\partial g}{\partial z}=k g\right.\right\}.
\]
\begin{definition}
The formal power series
\begin{equation}
p_{(f,G)}(t):=\sum_{k\ge 0} (\dim_\CC R_{(f,G),k}) t^k
\end{equation}
is called the {\em Poincar\'e series} corresponding to the pair $(f,G)$.
\end{definition}
It is easy to see that
the Poincar\'e series $p_{(f,G)}(t)$ is given by  
\[
p_{(f,G)}(t)=\frac{(1-t^{\frac{d}{c}})}{(1-t^{\frac{w_1}{c}})(1-t^{\frac{w_2}{c}})(1-t^{\frac{w_3}{c}})}
\]
and defines a rational function. In particular, the series $p_{(f,G_0)}(t)$ is the usual Poincar\'e series of the $\CC^\ast$-action on $f$ given by the reduced system of weights.

\begin{definition} Let $A_{f,G}=(\alpha_1, \ldots, \alpha_r)$ be the Dolgachev numbers of the pair $(f,G)$ and $g=g_{(f,G)}$. Define
\[\psi_{(f,G)}(t) := p_{(f,G)}(t) (1-t)^{2-2g}\prod_{i=1}^r \frac{1-t^{\alpha_i}}{1-t}. \]
\end{definition}

Note that, according to Theorem~\ref{thm:mirror}, 
\[ \phi(F,G_0^T)(t) = (1-t)^{2-2g}\prod_{i=1}^r \frac{1-t^{\alpha_i}}{1-t} \]
where $F$ is a polynomial of type $T_{\gamma'_1,\gamma'_2,\gamma'_3}$ with 
the Gabrielov numbers $(\gamma'_1,\gamma'_2,\gamma'_3)$ for $f^T$.

The polynomials $\psi_{(f,G_0)}(t)$ are listed in Table~\ref{phimon}. Here $c=c_f$ and we use the simplified notation
\[
i_1\cdot i_2\cdot\ \dots\ \cdot i_L\left/j_1\cdot j_2\cdot\ \dots\ \cdot j_M\right. .
\]
for a rational function of the form 
\[
\frac{\prod_{l=1}^L(1-t^{i_l})}{\prod_{m=1}^M(1-t^{j_m})}.
\]

\begin{table}[h]
\begin{center}
\begin{tabular}{|c||c|}
\hline
Type & $\psi_{(f,G_0)}(t)$ \\
\hline
\hline
I & $\left(\frac{p_1c_1}{c}\right)^{c_1}\cdot \left(\frac{p_2c_2}{c}\right)^{c_2} \cdot \left(\frac{p_3c_3}{c}\right)^{c_3}\cdot \frac{p_1p_2p_3}{c}\left/1^{c_1+c_2+c_3-2+2g}\cdot \frac{p_2p_3}{c} \cdot \frac{p_3p_1}{c} \cdot \frac{p_1p_2}{c} \right. $\\
 & $c_1 = (p_2, p_3), c_2 = (p_1, p_3), c_3 = (p_1, p_2)$ \\
\hline
II & $\left(\frac{p_1c_1}{c}\right)^{c_1}\cdot \left(\frac{p_3c_2}{p_2c}\right)^{c_2}\cdot  \frac{p_1p_3}{c}  \left/1^{c_1+c_2-1+2g} \cdot  \frac{p_3}{c} \cdot \frac{p_1p_3}{p_2c}  \right. $ \\
 & $c_1 = ( \frac{p_3}{p_2}, p_2-1), c_2 = (p_1, p_2)$ \\
\hline
III & $\left(\frac{p_1c_1}{c}\right)^{c_1}\cdot \frac{p_1p_2}{c}  \left/1^{c_1+2g}\cdot \frac{p_2}{c}  \right. $\\
 & $c_1 = ( q_2,q_3)$ \\
\hline
IV & $\left(\frac{p_3c_1}{p_2c}\right)^{c_1}\cdot \frac{p_3}{c}  \left/1^{c_1+2g}\cdot \frac{p_3}{p_1c} \right. $\\
 & $c_1 = (\frac{p_2}{p_1}, p_1-1)$ \\
\hline
V & $\frac{q_1q_2q_3+1}{c}  \left/1^{1+2g}\right. $\\
\hline
\end{tabular}
\end{center}
\caption{The polynomials $\psi_{(f,G_0)}(t)$}\label{phimon}
\end{table}

\begin{theorem} \label{thm:Poincare}
Let $f$ be an invertible polynomial with $c_f=c_{f^T}$ and $g=g_{(f,G_0)}=0$.
Then we have
\[ \psi_{(f,G_0)}(t) = \phi(f^T, G_0^T)(t).
\]
\end{theorem}

\begin{proof}[Proof of Theorem \ref{thm:Poincare}]
Let $c=c_f=c_{f^T}$ and $(\frac{w_1}{c},\frac{w_2}{c},\frac{w_3}{c};\frac{d}{c})$ be the reduced system of weights of $f$. Set $\widetilde{d}:=\frac{d}{c}$.
Table~\ref{phimon} shows that the polynomial $\psi_{(f,G_0)}(t)$ is of the form
\[ \psi_{(f,G_0)}(t) = \prod_{i |\widetilde{d}} (1-t^i)^{e(i)}, \quad e(i) \in \ZZ.
\]
Let $\omega_i$, $i=1, \ldots, \nu$, be the roots of this polynomial and put
\[ \Lambda_k := \omega_1^k + \cdots + \omega_\nu^k, \quad \mbox{for } k \in \ZZ_{\geq 0}. \]
Then 
\[ \Lambda_k = \sum_{m|k} m e(m). \]

On the other hand, let $(\widetilde{w}_1,\widetilde{w}_2,\widetilde{w}_3;\widetilde{d})$ be the reduced system of weights of $f^T$. Let $\widetilde{\omega}_i$, $i=1, \ldots, \widetilde{\nu}$, be the zeros of the polynomial $\phi_{(f^T, G_0^T)}(t)$ and put
\[ \widetilde{\Lambda}_k := \widetilde{\omega}_1^k + \cdots + \widetilde{\omega}_{\widetilde{\nu}}^k, \quad \mbox{for } k \in \ZZ_{\geq 0}. \]
Set $T^{\widetilde{d}}=\frac{\bar{t}}{t}$. Define a rational function $\chi(f^T, G_0^T)(T)$ as follows:
\begin{eqnarray*}
\chi(f^T, G_0^T)(T) & := & \sum_{g \in G_0^T} \chi_g(f^T, G_0^T)(T), \\
\chi_g(f^T, G_0^T)(T) & := & (-1)^{n_g} \left(\frac{\bar{t}}{t}\right)^{\frac{n_g}{2}} E_g(f^T, G_0^T)(t, \bar{t}).
\end{eqnarray*}
Then we have
\[ \widetilde{\Lambda}_k = \chi(f^T,G_0^T)\left({\bf e}\left[\frac{k}{\widetilde{d}}\right]\right)=
\sum_{g \in G_0^T} \chi_g(f^T,G_0^T)\left({\bf e}\left[\frac{k}{\widetilde{d}}\right]\right).
\]
From Theorem~\ref{thm:EGqh}, we can derive
\[ \chi_g(f^T,G_0^T)\left({\bf e}\left[\frac{k}{\widetilde{d}}\right]\right) =
(-1)^{n_g} \frac{1}{c}\sum_{h\in G_0^T}\left(\prod_{i:\frac{a_i}{p_i}\in\ZZ} \left( \delta((b_i + k) \widetilde{w}_i \,\mbox{mod}\, \widetilde{d}) \frac{\widetilde{d}}{\widetilde{w}_i} -1 \right) \right),
\]
where $\delta$ is the delta function, i.e., $\delta(0):=1$ and $\delta(x):=0$ for $x \neq 0$.
By a case by case study based on the proof of Theorem~\ref{thm:strange} and Table~\ref{phimon}, one can show that
\[ \Lambda_k = \widetilde{\Lambda}_k \mbox{ for all } k \in \ZZ_{\geq 0}. \]
This implies that $\psi_{(f,G_0)}(t)=\phi(f^T, G_0^T)(t)$.
\end{proof}

\begin{remark} Let $c_f=1$ and $c_{f^T}$ be arbitrary. Then it was shown in \cite[Theorem~22]{ET} that $\psi_{(f,G_0)}(t)$ is the characteristic polynomial of an operator $\tau$ such that $\tau^{c_{f^T}}$ is the monodromy of the singularity $f^T$. If $c_f$ is arbitrary and $c_{f^T}=1$, then one can derive from Table 10 and Table 11 of \cite{ET} that $\psi_{(f,G_f^{\rm fin})}(t)= \phi(f^T, (G_f^{\rm fin})^T)(t)=\phi(f^T, \{ {\rm id} \})(t)$.
\end{remark}

\section{Examples} \label{sect:Ex}
The classification of invertible polynomials with $\Delta(\Gamma_{(f^T,\{1\})}) < 0$ is given in 
Table~\ref{TabSpherical}. 
Note that $j_H=0$ for all subgroups $H$ of $G_0^T$.
In this case, we can describe explicitly the geometry of the proper transform of 
$(f^T)^{-1}(0)/H$ in the crepant resolution of $\CC^3/H$.
Consider the example $f^T(x,y,z)=x^2+y^3+z^4$. 
Then the crepant resolution of $\CC^3/(\ZZ/2\ZZ)$ is covered by two $\CC^3$'s. 
It is easily checked that the proper transform of $(f^T)^{-1}(0)/(\ZZ/2\ZZ)$ is smooth 
in one chart and in the other chart it is given by $x^2z+y^3+z^2=0$, which 
defines an $E_6$ singularity of Type II.
Thus, the geometry of Lagrangian vanishing cycles for $f^T(x,y,z)=x^2+y^3+z^4$ with 
the group $(\ZZ/2\ZZ)$ is equivalent to the one for $x^2z+y^3+z^2$ with the trivial group.
We denote this equivalence by $(E_{6}^I,\ZZ/2\ZZ)\simeq (E_{6}^{II},\{1\})$. 
Similarly, we get equivalences 
$(A_{2k-1}^I,\ZZ/2\ZZ)\simeq (D_{k+1}^{II},\{1\})$ where $\ZZ/2\ZZ$ acts as 
$(x,z)\mapsto (-x,-z)$ or $(y,z)\mapsto (-y,-z)$, 
$(D_{4}^I,\ZZ/3\ZZ)\simeq (D_{4}^{III},\{1\})$,
$(A_{4k-1}^{II},\ZZ/2\ZZ)\simeq (D_{2k+1}^{IV},\{1\})$,
$(A_{4k+1}^{II},\ZZ/2\ZZ)\simeq (D_{2k+2}^{III},\{1\})$, and 
$(D_k^{II},\ZZ/2\ZZ)\simeq (A_{2k-3}^{IV},\{1\})$.
In the example for $f^T(x,y,z)=x^2+y^2+z^{l}$ on which $\ZZ/2\ZZ$ acts as $(x,y)\mapsto (-x,-y)$, 
the proper transform of $(f^T)^{-1}(0)/(\ZZ/2\ZZ)$ in some chart 
is given by $x^2y+y+z^{l+1}=0$, which defines two $A_{l}$-singularities. 
Note that the polynomial $x^2y+y+z^{l+1}-xyz$ can be deformed to a polynomial 
of Type $T_{l+1,1,l+1}$ after the coordinate change $z\mapsto z+x$, 
which is of course compatible with Table~\ref{TabSpherical}, and also that 
the Coxeter--Dynkin diagram for a polynomial 
of Type $T_{l+1,1,l+1}$ naturally contains as a subdiagram the one for $A_l\times A_l$.
\begin{table}[h]
\begin{center}
\begin{tabular}{|c||c|c|c|c|c|}
\hline
{\rm Type} & $f^T(x,y,z)$ & $G_0^T$ & $\Gamma_{(f^T,\{1\})}$ & $\Gamma_{(f^T,G_0^T)}$ 
& {\rm Singularity}\\
\hline
\hline
I & $x^2+y^2+z^{2k+1}$, $k \geq 1$ & $\ZZ/2\ZZ$ & $2,2,2k+1$ & $2k+1,2k+1$& $A_{2k}$\\
 & $x^2+y^2+z^{2k}$, $k \geq 2$ & $(\ZZ/2\ZZ)^2$ & $2,2,2k$ & $k,k$ & $A_{2k-1}$\\
 & $x^2+y^3+z^3$ & $\ZZ/3\ZZ$ & $2,3,3$ & $2,2,2$ & $D_4$ \\
  & $x^2+y^3+z^4$ & $\ZZ/2\ZZ$ & $2,3,4$  & $2,3,3$ & $E_6$ \\
  & $x^2+y^3+z^5$ & $\{1\}$ & $2,3,5$ & $2,3,5$ & $E_8$ \\ 
\hline
II & $x^2+y^2+yz^k$, $k \geq 2$ & $\ZZ/2\ZZ$ & $2,2,2(k-1)$  & $2,2,k-1$ & $A_{2k-1}$\\
   & $x^2+y^{k-1} +yz^2$, $k \geq 4$ & $\ZZ/2\ZZ$ & $2,k-1,2$ & $k-1,k-1$ & $D_{k}$\\
   & $x^3+y^2+yz^2$ & $\{1\}$ & $3,2,3$ & $2,3,3$ & $E_6$\\
   & $x^2+y^3+yz^3$ & $\{1\}$ & $2,3,4$ & $2,3,4$ & $E_7$ \\
\hline
III & $x^2+zy^2+yz^{k+1}$, $k \geq 1$ & $\{1\}$ & $2,2,2k$ &  $2,2,2k$ & $D_{2k+2}$\\
\hline
IV & $x^l+xy+yz^k$, $k,l \geq 2$ & $\{1\}$ & $l, (k-1)l,1$ & $l, (k-1)l,1$ & $A_{kl-1}$\\
    & $x^2+xy^k+yz^2$, $k \geq 2$ & $\{1\}$ & $2,2,2k-1$ & $2,2,2k-1$ & $D_{2k+1}$ \\
\hline
V & $xy+y^kz+z^lx$, $k,l \geq 1$ & $\{1\}$ & $kl-k+1,1,k$ & $kl-k+1,1,k$ & $A_{kl}$\\
\hline
\end{tabular}
\end{center}
\caption{The cases $\Delta(\Gamma_{(f^T,\{1\})}) < 0$}\label{TabSpherical}
\end{table}

This is already different for invertible polynomials with $\Delta(\Gamma_{(f^T,\{1\})}) =0$.
As an example of such a polynomial we consider the polynomial $f(x,y,z)=f^T(x,y,z)= x^2+y^3+z^6$ defining a singularity of type $\widetilde{E}_8$. Here $G_0=\ZZ/6\ZZ$ and $G_f^{\rm fin}= \ZZ/2\ZZ \times \ZZ/3\ZZ \times \ZZ/6\ZZ$. The curve $C_{(f,G_0)}$ is a smooth elliptic curve. We introduce the following notation: Let $\alpha_1, \ldots , \alpha_r$ be positive integers with $\alpha_i \geq 2$ for $i=1, \ldots , r$. Denote by $\PP^1_{\alpha_1, \ldots , \alpha_r}$ the complex projective line with $r$ isotropic points of orders $\alpha_1, \ldots , \alpha_r$. Then ${\mathcal C}_{(f,G_f^{\rm fin})} = \PP^1_{2,3,6}$ and the mirror is the cusp singularity of type $T_{2,3,6}$. We have two non-trivial choices for $G$ with $G_0 \subset G \subset G_f^{\rm fin}$, namely $G/G_0= \ZZ/2\ZZ$ and $G/G_0=\ZZ/3\ZZ$. In the first case, the stack ${\mathcal C}_{(f,G)}$ is $\PP^1_{2,2,2,2}$ dual to the cusp singularity $T_{2,3,6}$ with the action of the group $\ZZ/3\ZZ$, in the second case it is  $\PP^1_{3,3,3}$ dual to the cusp singularity $T_{2,3,6}$ with the action of the group $\ZZ/2\ZZ$.

%
In \cite{EW85} an extension of Arnold's strange duality was considered which embraces on one hand series of bimodal hypersurface singularities and on the other hand, isolated complete intersection singularities (abbreviated ICIS) in $\CC^4$ (see also \cite{E99}). We show that this duality can be regarded as a special case of our mirror symmetry between orbifold curves and cusp singularities with group action. The eight  bimodal series start with six singularity classes in which there are weighted homogeneous representatives. (There are pairs of series with the same head.)
 In Table~\ref{TabBi}, invertible polynomials are chosen for these classes. The names are the names given by Arnold. We consider the pairs $(f, G)$ with $G=G_0$. We list the Dolgachev numbers of the pair $(f,G_0)$ in the second column. By Theorem~\ref{thm:mirror} they coincide with the Gabrielov numbers of the cusp singularities with group action $(f^T(x,y,z)-xyz, G_0^T)$. Moreover, they are also the Gabrielov numbers of the ICIS in $\CC^4$ which are dual to the bimodal singularities (see \cite{E99}). In some of these cases, the invariant part of the singularity $f^T$ under the action of the group $G_0^T$ is isomorphic to the corresponding ICIS. For each of these invertible polynomials $f$, the number $c_f$ is equal to 2. Hence the group $G_0^T$ is isomorphic to $\ZZ/2\ZZ$. It operates on $\CC^3$ by $(x,y,z) \mapsto (-x,-y,z)$. The invariant polynomials of this action are $X:=x^2, Y:=xy, Z:=z, W:=y^2$. They satisfy the relation $XW=Y^2$. If the equations of the invariant part of the singularity $f^T$ under the action of the group $G_0^T$ coincide with the defining equation of an ICIS  dual to the singularity $f$, then we list them on the right hand side of Table~\ref{TabBi}. 
\begin{table}[h]
\begin{center}
\begin{tabular}{|c|c|c|c||c|c|}
\hline
Series & Head & $A_{(f,G_0)}$  & $f$   & $(f^T)^{-1}(0)/G_0^T$ & Dual \\
\hline
$J_{3,k}$ & $J_{3,0}$ & $2,2,2,3$ & $x^6y+y^3+z^2$ & $\left\{ \begin{array}{c} XW-Y^2 \\ X^3 +YW+Z^2 \end{array}\right\}$  & $J_9'$ \\
$Z_{1,k}$ & $Z_{1,0}$ & $2,2,2,4$ & $x^5y + xy^3 +z^2$ & $\left\{ \begin{array}{c} XW-Y^2 \\ X^2Y +YW+Z^2 \end{array}\right\}$ & $J_{10}'$ \\
$Q_{2,k}$ & $Q_{2,0}$ & $2,2,2,5$ & $x^4y + y^3 + xz^2$ & $\left\{ \begin{array}{c} XW-Y^2 \\ X^2Z +YW+Z^2 \end{array}\right\}$  & $J_{11}'$\\
$W_{1,k}$ & $W_{1,0}$ & $2,2,3,3$ & $x^6+y^2+yz^2$  & $\left\{ \begin{array}{c} XW-Y^2 \\ X^3 +WZ+Z^2 \end{array}\right\}$ & $K_{10}'$ \\
$W^\sharp_{1,k}$ &                   & &  &  & $L_{10}$ \\
$S_{1,k}$ & $S_{1,0}$ & $2,2,3,4$ & $x^5+xy^2+yz^2$ & $\left\{ \begin{array}{c} XW-Y^2 \\ X^2Y +WZ+Z^2 \end{array}\right\}$  & $K_{11}'$ \\
$S^\sharp_{1,k}$ &   & & & & $L_{11}$ \\
$U_{1,k}$ & $U_{1,0}$ & $2,3,3,3$ & $x^3+xy^2+yz^3$  &   & $M_{11}$ \\
\hline
\end{tabular}
\end{center}
\caption{Duality between bimodal series and ICIS in $\CC^4$} \label{TabBi}
\end{table}

As another example we consider the invertible polynomial $f(x,y,z)=x^2+xy^3+yz^5$. The canonical weight system is $W_f=(15,5,5;30)$, so $c_f=5$ and the reduced weight system is $W_f^{\rm red}=(3,1,1;6)$. The genus of $C_{(f, G_0)}$ is equal to two, there are no exceptional orbits of the $\CC^\ast$-action. The Dolgachev numbers of the pair $(f,G_0)$ are $(5,5,5)$ and $G_0^T=\frac{1}{5}(1,3,1)$. This group has two elements of age 1. The singularity $f^T(x,y,z)-xyz$ is right equivalent to the cusp singularity $x^5+y^5+z^5-xyz$. We recover the example of \cite{Se}. Similarly, let $f(x,y,z)=x^3y+y^3z+z^3x$. Then the genus of  $C_{(f, G_0)}$ is equal to three, $G_0^T=\frac{1}{7}(1,2,4)$ and $f^T(x,y,z)-xyz$ is right equivalent to the cusp singularity $x^7+y^7+z^7-xyz$.

More generally, let $g$ be an integer with $g \geq2$ and consider the invertible polynomial $f(x,y,z)=x^{2g+1}+y^{2g+1}+z^{2g+1}$ together with the group $G:=\frac{1}{2g+1}(1,1,2g-1)$. Then the genus of the curve $C_{(f,G)}$ is equal to $g$ and we recover the examples of \cite{Ef}.


\end{document}